\documentclass[12pt, reqno]{amsart}
\usepackage{xcolor}
\usepackage{hyperref}

\definecolor{DarkNavy}{HTML}{002060}    
\definecolor{Maroon}{HTML}{800000}      
\definecolor{PineGreen}{HTML}{004D20}   

\hypersetup{
    colorlinks=true,
    linkcolor=blue,   
    citecolor=blue,  
    urlcolor=blue       
}

\usepackage{amsmath,amsthm}
\usepackage{amsfonts}
\usepackage{amssymb}
\usepackage{comment}
\usepackage{graphicx}
\usepackage{xcolor}
\usepackage[left, pagewise, edtable]{lineno}
\usepackage{enumitem}
\usepackage[margin=1in,footskip=0.25in]{geometry}
\allowdisplaybreaks[4]   
\providecommand{\U}[1]{\protect\rule{.1in}{.1in}}

\theoremstyle{plain}
\newtheorem{theorem}{Theorem}[section] 

\newtheorem{corollary}[theorem]{Corollary}

\newtheorem{lemma}[theorem]{Lemma}
\newtheorem{proposition}[theorem]{Proposition}

\theoremstyle{definition}
\newtheorem{definition}[theorem]{Definition}

\newtheorem{example}[theorem]{Example}

\newtheorem{assumption}[theorem]{Assumption}

\theoremstyle{plain}
\newtheorem{remark}[theorem]{Remark}

\title{On Continuous Wavelet Transform of Complex Radon Measures on $\mathbb{R}^{d}$}
\author[T. Banerjee]{Tathagata Banerjee}
\address{Tathagata Banerjee\\Department of Mathematics, Cochin University of Science and Technology, 682022, Cochin, Keralam, India}
\email{tathagatabanerjee85@gmail.com}

\author[V. Jayaprakash]{Vaishakh Jayaprakash}
\address{Vaishakh Jayaprakash\\Department of Mathematics, Cochin University of Science and Technology, 682022, Cochin, Keralam, India}
\email{vaishakhkj001@gmail.com}
\date{}

\begin{document}

\begin{abstract}
This paper investigates a necessary and sufficient condition for a measurable function on $\mathbb{R}^{\times} \times \mathbb{R}^{d}$ to be the continuous wavelet transform of a complex Radon measure on $\mathbb{R}^{d}$ under suitable regularity conditions on the wavelets. We investigate various integrability properties of the wavelet transform of complex Radon measures. We define a new norm on the collection of such functions and show that it is isometrically isomorphic to the Banach algebra of complex Radon measures under convolution. We show a wavelet inversion formula for measures, in a restricted sense. Under additional hypothesis, we provide a reconstruction formula for the measure of closed rectangles in $\mathbb{R}^d$ using wavelet transforms. We have included supporting examples to substantiate these assumptions. 
\end{abstract}

\maketitle

\section{Introduction}
The continuous wavelet transform has become an important mathematical tool with applications in signal and image processing, compressed sensing, and the theory of differential equations. It was initially defined to act on $L^{2}(\mathbb{R})$, introduced by Grossmann, Morlet and Paul in one of the foundational papers on wavelet analysis \cite{grossmann1985transforms}.  Subsequent research extended the theory from the one-dimensional setting to higher-dimensional Euclidean spaces, establishing a broader framework for wavelet analysis \cite{antoine19992,daubechies1992ten,Meyer1992Wavelets}. Further developments explored wavelet transforms and related constructions in other function spaces, including Lebesgue spaces $L^{p}$ and Sobolev spaces \cite{Pathak2009, rieder1990wavelet}.

 In this paper, we work with an extension of the continuous wavelet transform of functions to the framework of complex Radon measures on a general Euclidean space, $\mathbb{R}^d$. The continuous wavelet transform has been extended to spaces of distributions to address different problems in harmonic analysis and the study of regularity of functions. In \cite{holschneider1994fractal}, Holschneider introduced the continuous wavelet transform for certain Schwartz distributions to investigate fractal wavelet dimensions and the local regularity of functions.  Subsequently, Pathak \cite{pathak2004wavelet} studied the continuity properties of the continuous wavelet transform defined for $D_{L^{p}} $- type tempered distributions and functions in Besov spaces, Sobolev spaces, and Lizorkin-Triebel spaces. Furthermore, they established various boundedness results for the wavelet transform on Besov spaces and Lizorkin-Triebel spaces. A related investigation concerns the wavelet transform of compactly supported distributions, which has been carried out in \cite{ashurov2011continuous}. They investigated convergence properties of these transforms for distributions in Sobolev spaces with negative smoothness. A weak distributional sense inversion formula for the continuous wavelet transform defined on the subspace of all tempered distributions has been investigated in \cite{ pandey2019continuous1, pandey2019continuous} by using compactly supported smooth wavelets with non-vanishing moments for every positive order. We refer to \cite{zbMATH02233272, zbMATH05025041, saneva2010wavelet} for related works in the distributional wavelet transforms and references therein.
 
 This paper investigates the continuous wavelet transform applied to the space of complex Radon measures on Euclidean space with certain regularity conditions on the wavelet as in \ref{assumptions}. It is well-known that any finite Radon measure defined on the Euclidean space is naturally associated with a tempered distribution. Therefore, one can define the continuous wavelet transform of a complex Radon measure via the theory of distributions. The primary objective of this paper is to characterize those functions in the time-scale plane that arise as a wavelet transform of complex Radon measures after imposing suitable  regularity assumptions on the analyzing wavelets. To the best of our knowledge, such a characterization in the measure theoretic setting has not been investigated in the literature. 
 \subsection*{Main Contributions}
 Our main contributions in this paper can be summarized as follows: First, we establish a necessary and sufficient condition for functions to arise as wavelet transforms of complex Radon measures [Theorem \ref{main:thm1}], and we further demonstrate that this correspondence constitutes a linear isometric isomorphism between the respective spaces [Theorem \ref{main;thm 2}]. The embedding of the space of complex Radon measures in a weighted Lebesgue space is established [Proposition \ref{main;thm3}]. A weak-sense wavelet inversion formula for complex Radon measures is obtained using the classical wavelet inversion formula [Proposition \ref{measure re}]. Along the way, we give an alternate proof of the classical wavelet inversion formula for functions belonging to $C_c^{d+1}(\mathbb{R}^d)$ [Lemma \ref{reconstruction lemma}]. Secondly, we develop a comprehensive reconstruction procedure capable of recovering the underlying measure along with its precise values on closed rectangles [Proposition \ref{proplimit}]. Finally, we formulate a necessary condition under which the underlying measure can be successfully recovered through the wavelet transform of characteristic functions of the rectangles [Proposition \ref{passinglimit}].

 \subsection*{Organization of the Paper}
The remainder of the paper is organized as follows: 
In Section \ref{preli}, we fix the notation and recall necessary background required for this paper. The measure-theoretic notations and preliminaries are introduced in Section \ref{measurebackground}. The continuous wavelet transforms of functions and their fundamental properties are discussed in Section \ref{continuousbackground}. We present our main results and examples in Section \ref{mainresult}. The statement of the central result [Theorem \ref{main:thm1}] of this paper and its proof is provided in Section \ref{subsectionmain}. In Section \ref{conseq}, we discuss some of the consequences of the main result, that includes isometric correspondence and reconstruction procedure [Theorem \ref{main;thm 2}, Proposition \ref{proplimit}, Proposition \ref{passinglimit}]. We have also added necessary examples and counterexamples to show exactly why the integrability assumptions are required, as used in the results.

\section{Notations and Background}\label{preli}
\subsection{Measure-theoretic Background}\label{measurebackground} This section fixes the notation and conventions employed in the remainder of the paper.
For $1\leq p< \infty$, the space $L^{p}(\mathbb{R}^{d})$ consists of (equivalence classes of) measurable functions $f: \mathbb{R}^{d}\to \mathbb{C}$ such that 
\begin{equation*}
    \|f\|_{p}= \left(\int_{\mathbb{R}^{d}} \left|f(x)\right|^{p} dx\right)^{\frac{1}{p}}< \infty
\end{equation*} where the functions that agree almost everywhere with respect to Lebesgue measure $dx$ on $\mathbb{R}^{d}$ are identified. When $p=2$, it is naturally equipped with the inner product 
\begin{equation*}
    {\left\langle f, g \right \rangle}_{L^{2}(\mathbb{R}^{d})} =\int_{\mathbb{R}^{d}} f(x) \overline{g(x)} dx, \quad f,g \in L^{2}(\mathbb{R}^{d}).
\end{equation*} Let $X= \mathbb{R}^{\times} \times \mathbb{R}^{d}$ endowed with the measure $db \dfrac{da}{|a|^{d+1}}$. For a weight function $w: X\to (0,\infty)$, we define the weighted $L^{p}$ spaces, denoted by $L^{p}_{w}(X)$, as the collection of all functions $f: \mathbb{R}^{\times}\times \mathbb{R}^{d}\to \mathbb{C}$ such that 
\begin{equation*}
    \int_{\mathbb{R}^{\times }} \int_{\mathbb{R}^{d}} \left|f(a,b)\right|^{p} w(a,b) db \dfrac{da}{|a|^{d+1}}<\infty.
\end{equation*}
The convolution of two functions $f,g \in L^{1}(\mathbb{R}^{d})$ is given by 
\begin{equation*}
    f* g (x)= \int_{\mathbb{R}^{d}} f(x-y) g(y) dy = \int_{\mathbb{R}^{d}} f(y) g(x-y)dy. 
\end{equation*} The Fourier transform of $f\in L^{1}(\mathbb{R}^{d})$ is given by 
\begin{equation*}
    \hat{f}(\zeta) =\int_{\mathbb{R}^{d}} f(x) e^{-ix.\zeta} dx,
\end{equation*} and its inverse Fourier transform is 
\begin{equation*}
    f(x)= \dfrac{1}{(2\pi)^d} \int_{\mathbb{R}^{d}} \hat{f}(\zeta)e^{ix.\zeta} d\zeta.
\end{equation*}

The space $C(\mathbb{R}^{d})$ denotes the set of all complex-valued continuous functions on $\mathbb{R}^{d}$. The space $C_{0}(\mathbb{R}^{d})$ denotes the set of all $f\in C(\mathbb{R}^{d})$ that vanishes at infinity. It is a Banach space with respect to the supremum norm defined as 
\begin{equation*}
    \|f\|_{\infty}= \sup_{x\in \mathbb{R}^{d}} |f(x)|, \quad f\in C_{0}(\mathbb{R}^{d}).
\end{equation*} We will be using the following classes of functions in the forthcoming sections of this paper: 
\begin{align*}
     C_{c}^{k}(\mathbb{R}^{d})&=\{f\in C(\mathbb{R}^{d}): \quad \mbox{$f$ is compactly supported and $k$ times continuously differentiable}\}. \\
     C_{c}^{\infty}(\mathbb{R}^{d})&=\{f\in C(\mathbb{R}^{d}): \quad \mbox{$f$ is compactly supported and infinitely differentiable}\}.
\end{align*} Note that $C_{c}^{\infty}(\mathbb{R}^{d})\subset C_{c}^{k}(\mathbb{R}^{d})\subset C_{0}(\mathbb{R}^{d})$, and both the spaces $C_{c}^{k}(\mathbb{R}^{d})$ and $C_{c}^{\infty}(\mathbb{R}^{d})$ are dense in $C_{0}(\mathbb{R}^{d})$ with respect to the supremum norm $\|\cdot \|_{\infty} $ for any $k\in \mathbb{N}.$\\

Let $x=(x_1, x_2 , \cdots, x_{d})\in \mathbb{R}^{d}.$ A \textit{multi-index} $\alpha$ is a $d-$tuple 
\begin{equation*}
    \alpha= (\alpha_1,\alpha_2, \cdots, \alpha_d)\in \mathbb{N}_{\mathrm{0}}^{d}.
\end{equation*} Associated to a multi-index $\alpha$ we have the following notations: 
\begin{align*}
    |\alpha|&= \alpha_1 + \alpha_2 + \cdots+ \alpha_d.\\
     \alpha! &= \alpha_{1}! \alpha_{2}!\cdots \alpha_{d}!.\\
      x^{\alpha}&= x_{1}^{\alpha_1}\cdot x_{2}^{\alpha_{2}}\cdots x_{d}^{\alpha_{d}}.
\end{align*} We set 
\begin{equation*}
    D^{\alpha}= \dfrac{\partial ^{|\alpha|}}{\partial x_{1}^{\alpha_1} \partial x_{2}^{\alpha_1}\cdots \partial x_{d}^{\alpha_d}}.
\end{equation*}

Let $\mathcal{M}(\mathbb{R}^{d})$ denote the collection of all complex Radon measures on $\mathbb{R}^{d}.$ For $\mu, \nu \in \mathcal{M}(\mathbb{R}^{d})$, define
\begin{equation*}
    (\mu+ \nu)( E) = \mu(E) + \nu( E) 
\end{equation*} and 
\begin{equation*}
    (\alpha \mu)(E) = \alpha \mu(E)
\end{equation*} for any Borel set $E\subset \mathbb{R}^{d}$ and $\alpha \in \mathbb{C}$. Then $\mathcal{M} (\mathbb{R}^{d})$ forms a complex vector space with respect to these operations. For $\mu\in \mathcal{M}(\mathbb{R}^{d})$, we denote its total variation norm as $\|\mu\|_{TV}$. Consequently, $\mathcal{M}\left(\mathbb{R}^{d}, \|\cdot\|_{TV}\right)$ is a Banach space. The convolution of $\mu, \nu \in \mathcal{M}(\mathbb{R}^{d})$ is defined via the relation 
\begin{equation*}
    \int_{\mathbb{R}^{d}}  \phi  d (\mu * \nu) = \int_{\mathbb{R}^{d}} \int_{\mathbb{R}^{d}} \phi(x+y) d\mu(x) d \nu(y), \quad \mbox{for all} \quad \phi \in C_{0} (\mathbb{R}^{d}).
\end{equation*}  Observe that $\| \mu * \nu \|_{TV} \leq \|\mu\|_{TV} \|\nu \|_{TV}$ and $ \mu * \nu = \nu * \mu$. The dirac delta measure $\delta_{\mathrm{0}}$ concentrated at the origin acts as the identity with respect to the convolution, i.e., $\mu* \delta_{\mathrm{0}}= \mu = \delta_{\mathrm{0}}* \mu$ for all $\mu\in \mathcal{M}(\mathbb{R}^{d})$.  Thus, $\left(\mathcal{M}(\mathbb{R}^{d}), *\right)$ is a commutative unital Banach algebra. For $f\in L^{1}(\mathbb{R}^{d})$ and $\mu\in \mathcal{M}(\mathbb{R}^{d})$, one can define the function $\mu* f$ by 
\begin{equation*}
    \mu* f (x) =\int_{\mathbb{R}^{d}} f(x-y) d\mu(y).
\end{equation*}
\subsection{The Continuous Wavelet Transform} \label{continuousbackground}In this section, we recall the definition of the continuous wavelet transform of funtions and discuss some of its fundamental properties.\\

We call a non-zero function $\psi\in L^{2}(\mathbb{R}^{d})$ to be a \textit{wavelet}  if 
\begin{equation*}
    C_{\psi}:=\displaystyle \int_{\mathbb{R}^{\times}} \dfrac{\left|\hat{\psi}(a\zeta)\right|^{2}}{|a|}d a<\infty.
\end{equation*} For every $a\in \mathbb{R}^{\times}, b\in \mathbb{R}^{d}$, where $\mathbb{R}^{\times}$ denotes the set of all non-zero real numbers, we define the function $\psi_{a, b}$ as 
\begin{equation*}
    \psi_{a,b}(x) = |a|^{-\frac{d}{2}} \psi\left(\frac{x-b}{a}\right).
\end{equation*}
The \textit{continuous wavelet transform } of $f\in L^{2}(\mathbb{R}^{d})$ with respect to a wavelet $\psi\in L^{2}(\mathbb{R}^{d})$ is the function $\mathcal{W}_{\psi} f: \mathbb{R}^{\times} \times \mathbb{R}^{d}\to \mathbb{C}$ defined by 
\begin{equation*}
    \mathcal{W}_{ \psi} f(a, b)=\left\langle f, \psi_{a, b}\right\rangle_{L^{2}(\mathbb{R}^{d})} =|a|^{-\frac{d}{2}} \int_{\mathbb{R}^{d}} f(x) \overline{\psi\left(\frac{x-b}{a}\right)} d x.
\end{equation*} It is easy to see that $\mathcal{W}_{\psi}f(a,b)= \displaystyle (f* \psi_{a}) (b)$, where $\psi_{a}(x)= \displaystyle |a|^{-\frac{d}{2}} \psi\left(-\dfrac{x}{a}\right)$. Note that $|\mathcal{W}_{\psi} f(a,b)|\leq \|f\|_{2} \|\psi\|_{2}$. Furthermore, $\mathcal{W}_{\psi}: L^{2}(\mathbb{R}^{d}) \longrightarrow L^{2}\left(\mathbb{R}^{\times} \times \mathbb{R}^{d},  db \frac{ da}{|a|^{d+1}}\right)$ is multiple of a linear isometry, and gives rise to the weak-sense \textit{wavelet inversion formula} (or \textit{reconstruction formula}) for all $f \in L^{2}(\mathbb{R}^{d})$ by
\begin{equation*}
    f=\dfrac{1}{C_{\psi}}\int_{\mathbb{R}^{\times}} \int_{\mathbb{R}^{d}}\left\langle f, \psi_{a, b}\right\rangle_{L^{2}(\mathbb{R}^{d})} \psi_{a, b} db\frac{ d a}{|a|^{d+1}}
\end{equation*}
Equivalently, 
$$\langle f, h\rangle_{L^{2}(\mathbb{R}^{d})}=\dfrac{1}{C_{\psi}}\int_{\mathbb{R}^{\times}} \int_{\mathbb{R}^{d}}\left\langle f, \psi_{a, b}\right\rangle_{L^{2}(\mathbb{R}^{d})}\left\langle\psi_{a, b}, h\right\rangle_{L^{2}(\mathbb{R}^{d})} db \dfrac{ d a}{|a|^{d+1}}$$
for all $h \in L^{2}(\mathbb{R}^{d})$.
It is worthwhile to note that for all $f \in L^{2}((\mathbb{R}^{d})$,
$$
f(x)= \dfrac{1}{C_{\psi}}  \int_{\mathbb{R}^{\times}} \int_{\mathbb{R}^{d}}\left\langle f, \psi_{a, b}\right\rangle_{L^{2}(\mathbb{R}^{d})} \psi_{a, b}(x) db \dfrac{ d a}{|a|^{d+1}}
$$
for almost all $x \in \mathbb{R}^{d}$ (See \cite[Theorem 1]{rao1994application}). The isometry property of $\mathcal{W}_{\psi}$ also tells that the image $\mathcal{W}_{\psi} (L^{2}(\mathbb{R}^{d}))$ is a reproducing kernel Hilbert space. We refer the reader to \cite[Chapter 2]{daubechies1992ten} for a detailed exposition of continuous wavelet transforms.

\section{Main Results and Discussions}\label{mainresult}
\subsection{Main Result}\label{subsectionmain} In this section, we extend the notion of continuous wavelet transform of functions to the setting of complex Radon measures and prove the main theorem of this paper.

\begin{definition}\label{measure}
     The \textit{continuous wavelet transform } of $\mu \in \mathcal{M}\left(\mathbb{R}^{d}\right)$ with respect to a suitable wavelet $\psi$ is 
    \begin{equation}
        \mathcal{W}_{\psi} \mu(a,b) = |a|^{-\frac{d}{2}}\int_{\mathbb{R}^{d}} \overline{\psi \left(\frac{x-b}{a}\right)} d\mu(x)= \mu* \psi_{a}(b)
    \end{equation} for all $(a,b)\in \mathbb{R}^{\times} \times \mathbb{R}^{d}.$
\end{definition}

 With a slight abuse of notation, we use the notation $\mathcal{W}_{\psi}$ for continuous wavelet transform acting on both functions and measures. The intended meaning should be clear from the context.\\

Throughout this paper, we make the following assumption;

 \begin{assumption}\label{assumptions} Let $\psi$ denote a wavelet satisfying the following;
 \begin{enumerate}
     \item \textit{(Regularity)} $\psi \in C_{c}^{d+1}(\mathbb{R}^{d})$.
     \item \textit{(Vanishing moments) }\begin{equation*}
    \int_{\mathbb{R}^{d}} x^{\alpha} \psi(x) dx =0, \quad \mbox{for all} \quad 1\leq |\alpha|\leq d.
\end{equation*}
 \end{enumerate}
\end{assumption}
\begin{remark}
    Since $\psi \in C_{c}^{d+1} (\mathbb{R}^{d})$, then $\psi \in L^{1}(\mathbb{R}^{d})$ and $\hat{\psi}$ is continuous. The constant $C_{\psi}$ is finite only if $\hat{\psi}(0)=0$ or equivalently $\displaystyle \int_{\mathbb{R}^{d}} \psi(x) dx =0$. We refer to \cite{daubechies1992ten} for more details.
\end{remark}
We provide an example of a wavelet  satisfying Assumption \ref{assumptions} as follows;

\begin{example}\label{eg0} Take $d=1$.
    Take $\varphi: \mathbb{R}\to \mathbb{R}$ by 
\begin{equation*}
    \varphi(x)= \begin{cases}
        \left(1- x^{2}\right)^{5}, &\quad |x|<1\\
        0,& \quad |x|\geq 1.
    \end{cases}
\end{equation*} Then $\varphi\in C_{c}^{4}(\mathbb{R})$. Choose $\psi= \varphi''$. Explicitly, 
\begin{equation*}
    \psi(x)= \begin{cases}
        10\left(1- x^2\right)^3 \left(9x^2 -1\right), & \quad |x|<1\\
        0, & \quad |x|\geq 1
    \end{cases}.
\end{equation*} Observe that 
\begin{equation*}
    \int_{\mathbb{R}} \psi(x) dx = \int_{\mathbb{R}}\varphi'' (x) dx= \varphi'(1)-\varphi'(-1)=0
\end{equation*} and 
\begin{equation*}
    \int_{\mathbb{R}} x\psi(x) dx = \int_{\mathbb{R}} x \varphi''(x) dx = \left[x \varphi'(x)\right]_{-1}^{1}- \int_{\mathbb{R}} \varphi'(x) dx =0.
\end{equation*} Thus, $\psi$ is a wavelet that satisfies Assumption \ref{assumptions}. 
\end{example}

\begin{definition} Let $\psi$ denote a wavelet satisfying the Assumption \ref{assumptions}.
    Let $g: \mathbb{R}^{\times}\times \mathbb{R}^{d}\to \mathbb{C}$ denote a measurable function satisfying the following conditions;
    \begin{enumerate}
\item[(a)] For all $(a,b)\in \mathbb{R}^{\times} \times \mathbb{R}^{d}$, 

$$g(a, b)= \dfrac{1}{C_{\psi}}  \int_{\mathbb{R}^{\times}} \int_{\mathbb{R}^{d}} g(r, s){\langle \overline{\psi_{a, b}}, \psi_{r, s}\rangle}_{L^{2}(\mathbb{R}^d)}ds\dfrac{ d r}{|r|^{d+1}}.$$

\item[(b)] There exists a constant $C>0$ such that for all $f \in C_{c}^{d+1}(\mathbb{R}^{d})$,
$$
\left|\int_{\mathbb{R}^{\times}} \int_{\mathbb{R}^{d}} g(a, b) \mathcal{W}_{\psi} f(a, b) db \dfrac{ d a}{|a|^{d+1}}\right| \leq C\|f\|_{\infty} .
$$
\end{enumerate} Denote $\mathcal{F}_{\psi}$ to be the set of all measurable functions $g$ satisfying the conditions $(a)$ and $(b)$ above.

For $g\in \mathcal{F}_{\psi}$, define 
        \begin{equation}\label{norm}
            \|g\|_{\psi}:= \sup \left\{\dfrac{1}{C_{\psi}}\left| \int_{\mathbb{R}^{\times}} \int_{\mathbb{R}^{d}} g(a,b) \mathcal{W}_{\psi} f(a,b) db \dfrac{da}{|a|^{d+1}}\right|: f\in C_{c}^{d+1}(\mathbb{R}^{d}), \quad \|f\|_{\infty}\leq 1\right\}.
        \end{equation}
\end{definition}

\begin{proposition}\label{image}
    The space $\left(\mathcal{F}_{\psi}, \|\cdot\|_{\psi}\right)$ is a normed space.
\end{proposition}
\begin{proof}
It is trivial to verify that $\mathcal{F}_{\psi}$ is a vector space. The well-definedness of $\|\cdot \|_{\psi}$ follows from condition $(b)$ of Theorem \ref{main:thm1}. The triangle inequality, scalar compatibility and positivity follows easily. \\

    Assume $\|g\|_{\psi}=0$. Then 
\begin{equation}\label{new}
     \int_{\mathbb{R}^{\times}} \int_{\mathbb{R}^{d}} g(a,b) \mathcal{W}_{\psi} f(a,b) db \dfrac{ da}{|a|^{d+1}}=0
\end{equation} for all $f\in C_{c}^{d+1}(\mathbb{R}^{d})$ with $\|f\|_{\infty}\leq 1$. Fix $(r,s)\in \mathbb{R}^{\times} \times \mathbb{R}^{d}.$ Define $\displaystyle h= \displaystyle \dfrac{|r|^{\frac{d}{2}}}{\|\psi\|_{\infty}} \overline{\psi_{r,s}}$. Then $\|h\|_{\infty}\leq 1$ and hence $h$ satisfies Equation \ref{new}. This together with condition $(a)$ of Theorem \ref{main:thm1} yields that 
\begin{align*}
    g(r,s)= \dfrac{1}{C_{\psi}}\int_{\mathbb{R}^{\times}} \int_{\mathbb{R}^{d}} g(a,b) {\langle \overline{\psi_{r,s}}, \psi_{a,b} \rangle}_{L^{2}(\mathbb{R}^{d})} db \dfrac{ da}{|a|^{d+1}}=0.
\end{align*} This shows that $g=0$. This shows that $\|\cdot \|_{\psi}$ is a norm on $\mathcal{F}_{\psi}.$ This completes the proof.    
\end{proof}
 First, we prove the following technical lemma: 
\begin{lemma}\label{main:lemma1}
    Let $\psi$ denote a wavelet satisfying the Assumption \ref{assumptions}. Then
    \begin{equation*}
        \int_{\mathbb{R}^{\times}} \int_{\mathbb{R}^{d}} |a|^{-\frac{d}{2}}\left|\mathcal{W}_{\psi} f (a,b)\right| db \dfrac{ da}{|a|^{d+1}}< \infty
    \end{equation*} for any $f\in C_{c}^{d+1}(\mathbb{R}^{d}).$
\end{lemma}
\begin{proof}
     We provide the arguments for $0<|a|\leq 1$ and $|a|>1$ separately.\\

Assume that $|a|>1$. Note that
\begin{align*}
\int_{\mathbb{R}^{d}} |a|^{-\frac{d}{2}}\left|\mathcal{W}_{\psi} f(a, b)\right| d b & =|a|^{-\frac{d}{2}}\int_{\mathbb{R}^{d}}\left|\mathcal{W}_{\psi} f(a, b)\right| d b \\
& =|a|^{-\frac{d}{2}} \int_{\mathbb{R}^{d}}\left|\int_{\mathbb{R}^{d}} f(x) \overline{\psi_{a, b}}(x)d x\right| d b \\
&\leq |a|^{-d}  \int_{\mathbb{R}^{d}} \int_{\mathbb{R}^{d}} \left| f(x)\right| \left| {\psi \left(\dfrac{x-b}{a}\right)} \right|dx db\\
&=  \int_{\mathbb{R}^{d}}  \int_{\mathbb{R}^{d}}   \left|f(ay+b)\right| \left|{\psi(y)}\right| dy db \quad \left(\mbox{Put $y= \dfrac{x-b}{a}$, $dx= |a|^{d} dy$}\right)\\
& =  \int_{\mathbb{R}^{d}} \left[\int_{\mathbb{R}^{d}} \left|f(ay+b)\right| db\right] |{\psi(y)}|dy \quad \left(\mbox{By Fubini-Tonelli theorem}\right)\\
&= \int_{\mathbb{R}^{d}} \left[\int_{\mathbb{R}^{d}} |f(z)| dz\right]|\psi(y)| dy, \quad \left(\mbox{Put $z= ay+ b, dz=db$}\right)\\
&=  \|f\|_{1} \|\psi\|_{1}.
\end{align*} Set $C_1 = \|\psi\|_{1} \|f\|_{1}$. Thus, we have shown that  
$$
\int_{\mathbb{R}^{d}} |a|^{-\frac{d}{2}} \left|\mathcal{W}_{\psi} f(a,b)\right| db \leq C_{1} .
$$ Integrate this with respect to $\dfrac{ da}{|a|^{d+1}}$ over the domain $\{a\in \mathbb{R}^{\times}: |a|>1\}$, we get 
\begin{align*}
    \int_{|a|>1} \left[\int_{\mathbb{R}^{d}} |a|^{-\frac{d}{2}} \left|\mathcal{W}_{\psi} f(a,b)\right| db\right] \frac{da}{|a|^{d+1}}\leq C_{1} \int_{|a|>1} \dfrac{1}{|a|^{d+1}} da < \infty. 
\end{align*}

Now consider the case when $0< |a|\leq 1$. Since $f\in C_{c}^{d+1}(\mathbb{R}^{d})$, then the Taylor series of $f(ay+b)$ around $b\in \mathbb{R}^{d}$ truncated at the $d+1$-th power term is 
\begin{equation*}
    f(ay+b)= \sum_{|\alpha|\leq d} \dfrac{D^{\alpha} {f (b)}}{\alpha!} a^{|\alpha|} y^{\alpha} + a^{d+1} \sum_{|\alpha|=d+1} \dfrac{D^{\alpha} f(b+ a\theta y)}{\alpha!} y^{\alpha}
\end{equation*} for some $\theta \in (0,1).$
 Observe that 
\begin{align*}
    \mathcal{W}_{\psi}f(a,b)&=|a|^{-\frac{d}{2}} \int_{\mathbb{R}^{d}} f(x) \overline{\psi\left(\frac{x-b}{a}\right)} dx\\
    &= |a|^{\frac{d}{2}}  \int_{\mathbb{R}^{d}} f(ay+b) \overline{\psi(y)} dy \quad \left( \mbox{Put} \quad y=\frac{x-b}{a}\right)\\
    &= |a|^{\frac{d}{2}} \int_{\mathbb{R}^{d}} \left[\sum_{|\alpha|\leq d} \dfrac{D^{\alpha} {f (b)}}{\alpha!} a^{|\alpha|} y^{\alpha} + a^{d+1} \sum_{|\alpha|=d+1} \dfrac{D^{\alpha} f(b+ a\theta y)}{\alpha!} y^{\alpha}\right] \overline{\psi(y)} dy \\
    &=|a|^{\frac{d}{2}} \left[ \sum_{|\alpha|\leq d} \dfrac{D^{\alpha} f(b)}{\alpha!} a^{|\alpha|}\int_{\mathbb{R}^{d}} y^{\alpha} \overline{\psi(y)} dy + a^{d+1} \sum_{|\alpha|= d+1} \dfrac{D^{\alpha} f (b+ a\theta y)}{\alpha!} \int_{\mathbb{R}^{d}} y^{\alpha} \overline{\psi(y)} dy\right]\\
    &=  |a|^{\frac{d}{2}} a^{d+1} \sum_{|\alpha|= d+1} \dfrac{D^{\alpha} f (b+ a\theta y)}{\alpha!} \int_{\mathbb{R}^{d}} y^{\alpha} \overline{\psi(y)} dy \quad \left(\mbox{Using vanishing moments}\right).
\end{align*} Take absolute value on both sides, 

\begin{align*}
    \left|\mathcal{W}_{\psi} f(a,b)\right| \leq |a|^{\frac{d}{2}+ d+1} \sum_{|\alpha|=d+1} \dfrac{\left| D^{\alpha} f(b+ a\theta y)\right|}{\alpha!} \int_{\mathbb{R}^{d}} \left|y^{\alpha}\right| |\psi(y)| dy.
\end{align*}
Integrate with respect to $db$, 

\begin{align*}
   \int_{\mathbb{R}^{d}} \left|\mathcal{W}_{\psi} f(a,b)\right| db &\leq |a|^{\frac{d}{2}+ d+1} \int_{\mathbb{R}^{d}} \left[ \sum_{|\alpha|=d+1} \dfrac{\left| D^{\alpha} f(b+ a\theta y)\right|}{\alpha!} \int_{\mathbb{R}^{d}} \left|y^{\alpha}\right| |\psi(y)| dy\right] db\\
   &=|a|^{\frac{d}{2}+d+1} \sum_{|\alpha|= d+1} \dfrac{1}{\alpha!} \int_{\mathbb{R}^{d}} \int_{\mathbb{R}^{d}} \left|y^{\alpha}\right| \left|D^{\alpha} f(b+ a\theta y)\right| |\psi(y)| dy db\\
   &= |a|^{\frac{d}{2}+d+1} \sum_{|\alpha|=d+1} \dfrac{1}{\alpha!} \int_{\mathbb{R}^{d}}\left[\int_{\mathbb{R}^{d}} \left|D^{\alpha}f(b+ a\theta y)\right|db\right] \left|y^{\alpha}\right| \left|\psi(y)\right|dy \left(\mbox{Using Fubini-Tonelli}\right) \\
   &= |a|^{\frac{d}{2}+d+1} \sum_{|\alpha|=d+1} \dfrac{1}{\alpha!} \int_{\mathbb{R}^{d}}\left[\int_{\mathbb{R}^{d}} \left|D^{\alpha}f(z)\right|dz\right] \left|y^{\alpha}\right| \left|\psi(y)\right|dy \\
   &= |a|^{\frac{d}{2}+ d+1} \sum_{|\alpha|=d+1} \dfrac{\|D^{\alpha}f\|_{1}}{\alpha!} \int_{\mathbb{R}^{d}} \left|y^{\alpha}\right| \left|\psi(y)\right| dy. 
\end{align*}
Since $\psi$ is compactly supported, then there exists $R>0$ such that $|y_{i}|\leq R$ for all $i=1,2,\cdots, d$. Therefore, 
\begin{align*}
    \int_{\mathbb{R}^{d}} \left|y^{\alpha}\right| \left|\psi(y)\right| dy = \int_{\mathbb{R}^{d}} |y_{1}|^{\alpha_1} |y_{2}|^{\alpha_2}\cdots |y_{d}|^{\alpha_d} |\psi(y)| dy =
    \leq R^{|\alpha|} \|\psi\|_{1}= R^{d+1}  \|\psi\|_{1}.
\end{align*} Note that the cardinality of the set $\{(\alpha_1, \alpha_2, \cdots, \alpha_d)\in \mathbb{N}_{\mathrm{0}}^{d}: |\alpha|=d+1\}$ is finite. Since $f\in C_{c}^{d+1}(\mathbb{R}^{d})$, then $D^{\alpha} f \in L^{1}(\mathbb{R}^{d})$. Thus, 
\begin{equation*}
     \int_{\mathbb{R}^{d}} \left|\mathcal{W}_{\psi} f(a,b)\right| db\leq |a|^{\frac{d}{2}+d+1} R^{d+1} \|\psi\|_{1} \sum_{|\alpha|=d+1} \dfrac{\|D^{\alpha} f\|_{1}}{\alpha!}.
\end{equation*} Set $C_{2}= \displaystyle R^{d+1} \|\psi\|_{1} \sum_{|\alpha|=d+1} \dfrac{\|D^{\alpha} f\|_{1}}{\alpha!}$, and hence $C_{2}<\infty.$
 Multiplying with $|a|^{-\frac{d}{2}}$ gives 
\begin{equation*}
    |a|^{-\frac{d}{2}} \int_{\mathbb{R}^{d}} \left|\mathcal{W}_{\psi} f(a,b)\right| db \leq C_{2} |a|^{d+1}.
\end{equation*}
Integrating with respect to $\dfrac{da}{|a|^{d+1}}$ over $\{a\in \mathbb{R}^{\times} :0<|a|\leq 1\}$ yields that 
\begin{equation*}
    \int_{0<|a|\leq 1} |a|^{-\frac{d}{2}} \left[\int_{\mathbb{R}^{d}} \left|\mathcal{W}_{\psi} f(a,b)\right| db\right] \frac{da}{|a|^{d+1}} \leq C_{2} \int_{0<|a|\leq 1}  da =2C_{2}< \infty.
\end{equation*}
As a consequence, 
\begin{align*}
    \int_{\mathbb{R}^{\times}} \int_{\mathbb{R}^{d}} |a|^{-\frac{d}{2}} \left|\mathcal{W}_{\psi} f(a,b)\right| db \dfrac{ da}{|a|
    ^{d+1}}&=\int_{|a|>1} \int_{\mathbb{R}^{d}} |a|^{-\frac{d}{2}} \left| \mathcal{W}_{\psi} f(a,b)\right|  db \dfrac{ da }{|a|^{d+1}} \\ & + \int_{0<|a|\leq 1} \int_{\mathbb{R}^{d}} |a|^{-\frac{d}{2}} \left| \mathcal{W}_{\psi} f(a,b)\right| db \dfrac{ da }{|a|^{d+1}}< \infty.
   \end{align*} This completes the proof.
\end{proof} 

In the next Lemma we use two different wavelets $\psi^1$ and $\psi^2$ satisfying the Assumption \ref{assumptions}. This gives rise to a new constant $C_{\psi^{1}, \psi^{2}}:= \displaystyle \int_{\mathbb{R}^{\times}} \dfrac{\hat{\psi^{1}}(a \zeta) \overline{\hat{\psi^{2}}( a\zeta)}}{\left|a\right|}d\zeta< \infty$, where $\hat{\psi^{1}}, \hat{\psi^{2}}$ is the Fourier transform of the wavelets $\psi^1$ and $\psi^2$ respectively. We use Lemma \ref{main:lemma1}, to estimate a certain integral that arises in the following Lemma \ref{reconstruction lemma}. This gives an alternate proof of the classical wavelet inversion formula for functions in $C_c^{d+1}(\mathbb{R}^d)$.
\begin{lemma}\label{reconstruction lemma}
    Let $\psi^1, \psi^2$ denote wavelets satisfying the Assumption \ref{assumptions}. Then for all $f\in C_{c}^{d+1}(\mathbb{R}^{d})$
    \begin{equation}\label{allx}
        f(x)= \dfrac{1}{C_{\psi^{1}, \psi^{2}}}\int_{\mathbb{R}^{\times}} \int_{\mathbb{R}^{d}} {\left \langle f, \psi^{1}_{r,s} \right \rangle}_{L^{2}(\mathbb{R}^{d})} \psi^{2}_{r,s}(x) ds \dfrac{ dr}{|r|^{d+1}},
    \end{equation} for all $x\in \mathbb{R}^{d}$. 
\end{lemma}
\begin{proof}
    Let 
    $$h(x) =\dfrac{1}{C_{\psi^{1}, \psi^{2}}}    \int_{\mathbb{R}^{\times}} \int_{\mathbb{R}^{d}}\left\langle f, \psi^1_{a, b}\right\rangle_{{L}^{2}(\mathbb{R}^{d})} \psi^2_{a, b}(x) db \frac{ d a}{|a|^{d+1}}.$$
    We claim that $h$ is continuous on $\mathbb{R}^{d}$. Let $x_{0} \in \mathbb{R}^{d}$ and let $\left(x_{n}\right)$ be a sequence converging to $x_{0}.$ It is enough to show that $\displaystyle \lim _{n \rightarrow \infty} h\left(x_{n}\right)=h\left(x_{0}\right)$. Define $$M_{n}(a, b)=\left\langle f, \psi^1_{a, b}\right\rangle_{L^{2}( \mathbb{R}^{d})} \psi^2_{a, b}\left(x_{n}\right)$$ for all $n\in \mathbb{N}$.
The continuity of $\psi^2_{a, b}$ implies that
$$
M(a, b):=\lim _{n \rightarrow \infty} M_{n}(a, b)=\left\langle f, \psi^1_{a, b}\right\rangle_{L^{2}(\mathbb{R}^{d})} \psi^2_{a, b}\left(x_{0}\right).
$$ 
Note that for all $n \in \mathbb{N}$
\begin{align*}
\left|M_{n}(a, b)\right| & =\left|\left\langle f, \psi^1_{a, b}\right\rangle_{L^{2}(\mathbb{R}^{d})} \psi^2_{a, b}\left(x_{n}\right)\right| \\
& =\left|\mathcal{W}_{\psi^1} f(a, b)\right|\left|\psi^2_{a, b}\left(x_{n}\right)\right| \\
& \leq |a|^{-\frac{d}{2}}\|\psi^2\|_{\infty}|\mathcal{W}_{\psi^1} f(a, b)|.
\end{align*}
The last step follows from $\psi^2$ being continuous and compactly supported and hence bounded. Hence we show that the sequence of functions $\left(M_{n}\right)_{n \in \mathbb{N}}$ is pointwise bounded and the bound is independent of $n \in \mathbb{N}$. Also it follows from Lemma \ref{main:lemma1} that this bound is integrable with respect to $db \dfrac{ da}{|a|^{d+1}}$. 
Therefore by dominated convergence theorem, 
\begin{equation*}
    \lim_{n\to \infty} h(x_{n}) = \lim_{n\to \infty} \int_{\mathbb{R}^{\times}} \int_{\mathbb{R}^{d}} M_{n}(a,b)db\dfrac{ da}{|a|^{d+1}} = \int_{\mathbb{R}^{\times}} \int_{\mathbb{R}^{d}} M(a,b) db\dfrac{da }{|a|^{d+1}}= h(x_{0}).
\end{equation*} This proves that $h$ is continuous at $x_{0}\in \mathbb{R}^{d}$. Since $x_{0}$ is arbitrary, then $h$ is continuous on $\mathbb{R}^{d}$. We know that $f(x)=h(x)$ for almost all $x\in \mathbb{R}^{d}$. Then continuity of $f$ implies that $f(x)= h(x)$ for all $x\in \mathbb{R}^{d}$. This completes the proof.
\end{proof}

In particular, let $\psi$ denote a wavelet satisfying the hypothesis of Theorem \ref{main:thm1}. Then by choosing $\psi^1= \psi$ and $\psi^2=\overline{\psi}$ in Lemma \ref{reconstruction lemma} we get the following reconstruction formula
\begin{equation}
  f(x)= \dfrac{1}{C_{\psi}}\int_{\mathbb{R}^{\times}} \int_{\mathbb{R}^{d}} {\left \langle f, \psi_{r,s} \right \rangle}_{L^{2}(\mathbb{R}^{d})} \overline{\psi_{r,s}}(x) ds \dfrac{ dr}{|r|^{d+1}}, \label{formula}  
\end{equation} for all $x\in \mathbb{R}^d$.
\begin{remark}
    Note that the above lemma \ref{reconstruction lemma} in dimension one, i.e., $\mathbb{R}$, is true for bounded integrable functions at the continuity points under mild regularity assumptions on $\psi^1$ and $\psi^2$, which is provided in \cite[Proposition 2.4.2]{daubechies1992ten}, and the Lemma \ref{reconstruction lemma} is a special case. In particular, if the wavelets $\psi^{1}$ and $\psi^{2}$ satisfy $\psi^{1}, \psi^{2} \in L^{1}(\mathbb{R})$, $\psi^{1}$ is differentiable , $x\psi_2 \in L^{1}(\mathbb{R})$, $(\psi^{2})'\in L^{2}(\mathbb{R})$ and  $\hat{\psi^{1}}(0)=0= \hat{\psi^{2}}(0)$, then for all bounded $f\in L^{2}(\mathbb{R})$ Equation \ref{allx} holds pointwise at the points where $f$ is continuous.
\end{remark} 
 The following Proposition is a wavelet inversion formula for measures, applicable for functions in $C_{c}^{d+1}(\mathbb{R}^{d}).$
\begin{proposition}\label{measure re}
    Let $\psi$ denote a wavelet satisfying the Assumption \ref{assumptions} and let $\mu \in \mathcal{M}(\mathbb{R}^{d})$. Then 
    \begin{equation}\label{measurereconstruction}
            \int_{\mathbb{R}^{d}} f(x) d\mu(x) = \dfrac{1}{C_{\psi}}\int_{\mathbb{R}^{\times}} \int_{\mathbb{R}^{d}} \mathcal{W}_{\psi} \mu(a,b) \mathcal{W}_{\psi} f(a,b) db \dfrac{da}{|a|^{d+1}}
        \end{equation} for all $f\in C_{c}^{d+1}(\mathbb{R}^{d}).$ 
\end{proposition}

\begin{proof}
    Let $g(r,s)= \mathcal{W}_{\psi} \mu(r,s)$. Observe that for $f\in C_{c}^{d+1} (\mathbb{R}^{d}),$
\begin{align*}
    \int_{\mathbb{R}^{\times}} \int_{\mathbb{R}^{d}} \int_{\mathbb{R}^{d}} \left| \psi_{a,b}(x) \mathcal{W}_{\psi} f(a,b)\right|  d\mu(x) db \dfrac{da}{|a|^{d+1}}&= \int_{\mathbb{R}^{\times}} \int_{\mathbb{R}^{d}} \int_{\mathbb{R}^{d}}  |a|^{-\frac{d}{2}} \left|\psi \left(\dfrac{x-b}{a}\right)\right| \left|\mathcal{W}_{\psi} f(a,b)\right| d\mu(x) db \dfrac{da }{|a|^{d+1}}\\
    &\leq \|\psi\|_{\infty} \|\mu\|_{TV} \int_{\mathbb{R}^{\times}} \int_{\mathbb{R}^{d}} |a|^{-\frac{d}{2}}\left|\mathcal{W}_{\psi} f(a,b)\right| db \dfrac{ da}{|a|^{d+1}}< \infty.
\end{align*} The finiteness of right hand side follows from Lemma \ref{main:lemma1}. Thus, 
\begin{align*}
      &\dfrac{1}{C_{\psi}}      \int_{\mathbb{R}^{\times}} \int_{\mathbb{R}^{d}} g(r,s) {\langle f, \psi_{r,s} \rangle}_{L^{2}(\mathbb{R}^{d})} ds \dfrac{ dr }{|r|^{d+1}}\\&=\dfrac{1}{C_{\psi}}\int_{\mathbb{R}^{\times}} \int_{\mathbb{R}^{d}} \mathcal{W}_{\psi}\mu(r,s) {\langle f, \psi_{r,s} \rangle}_{L^{2}(\mathbb{R}^{d})} ds \dfrac{ dr }{|r|^{d+1}}\\
            & =\dfrac{1}{C_{\psi}}   \int_{\mathbb{R}^{\times}} \int_{\mathbb{R}^{d}} \left[\int_{\mathbb{R}^{d}} \overline{\psi_{r,s}}(x) d\mu(x)\right] {\langle f, \psi_{r,s} \rangle}_{L^{2}(\mathbb{R}^{d})} ds \dfrac{dr }{|r|^{d+1}}\\
            &=   \int_{\mathbb{R}^{d}} \left[ \dfrac{1}{C_{\psi}}  \int_{\mathbb{R}^{\times}} \int_{\mathbb{R}^{d}} \overline{\psi_{r,s}}(x) {\langle f, \psi_{r,s} \rangle}_{L^{2}(\mathbb{R}^{d})} ds \dfrac{dr}{|r|^{d+1}} \right] d\mu(x) \quad  (\mbox{Using Fubini-Tonelli theorem})\\
            &= \int_{\mathbb{R}^{d}} f(x) d\mu(x) \quad (\mbox{Using Equation \ref{formula}}).
            \end{align*}  This completes the proof.
\end{proof}

We state the main result of this paper in the following Theorem \ref{main:thm1}. It is in the same vein as the Bochner theorem in the area of Fourier analysis, where we get a characterization of which functions arise as the Fourier transform of measures \cite{folland2016course}.
\begin{theorem} \label{main:thm1}
Let $\psi$ denote a wavelet satisfying the Assumption \ref{assumptions}. Let $g: \mathbb{R}^{\times} \times \mathbb{R}^{d} \rightarrow \mathbb{C}$ denote a measurable function. Then $g=\mathcal{W}_{\psi} \mu$ for some $\mu \in \mathcal{M}(\mathbb{R}^{d})$ if and only if $g\in \mathcal{F}_{\psi}$.
\end{theorem}
\begin{proof}
    $(\impliedby)$ Let $g: \mathbb{R}^{\times} \times \mathbb{R}^{d} \rightarrow \mathbb{C}$ be the measurable function that satisfies the conditions $(a)$ and $(b).$  We need to construct a complex Radon measure $\mu$ such that $g=\mathcal{W}_{\psi} \mu$. \\
  
  Define $T: C_{c}^{d+1}(\mathbb{R}^{d}) \rightarrow \mathbb{C}$ by
\begin{equation}\label{abc}
T(f)=\dfrac{1}{C_{\psi}}\int_{\mathbb{R}^{\times}} \int_{\mathbb{R}^{d}} g(r,s) \mathcal{W}_{\psi}  f(r,s) ds \frac{d r}{|r|^{d+1}}. 
\end{equation} 
The linearity of $\mathcal{W}_{\psi}$ gives the linearity of $T$. Using condition $(b)$,
$$
|T(f)|=\left|\dfrac{1}{C_{\psi}}\int_{\mathbb{R}^{\times}} \int_{\mathbb{R}^{d}} g(r,s)  \mathcal{W}_{\psi} f(r, s) ds \frac{d r}{|r|^{d+1}}\right| \leq \dfrac{C}{C_{\psi}}\| f \|_{\infty}. 
$$ By taking supremum over the unit ball, we have $\|T\|_{\infty} \leq  \dfrac{C}{C_{\psi}} .$ Thus, $T$ is a bounded linear functional on $C_{c}^{d+1}(\mathbb{R}^{d})$. 
The map $T$ can be extended uniquely to $C_{0}(\mathbb{R}^{d})$ as a bounded linear functional $\tilde{T}$ using the density of $C_{c}^{d+1}(\mathbb{R}^{d})$ in $C_{0}(\mathbb{R}^{d})$ in the supremum norm. Note that $\tilde{T}(f)=T(f)$ for all $f \in C_{c}^{d+1} (\mathbb{R}^{d})$. 

By applying Riesz representation theorem to $\tilde{T}$, there exists a unique $\mu\in \mathcal{M}(\mathbb{R}^{d})$ such that
\begin{equation}\label{functionalrep}
\tilde{T}(f)=\int_{\mathbb{R}^{d}} f(x) d \mu(x) 
\end{equation} for all $f\in C_{0}(\mathbb{R}^{d})$.

Note that since $\psi \in C_{c}^{d+1}(\mathbb{R}^{d})$, then $\psi_{a, b} \in C_{c}^{d+1}(\mathbb{R}^{d})$ for all $(a, b) \in \mathbb{R}^{\times} \times \mathbb{R}^{d}$. Substitute $f=\overline{\psi_{a, b}}$ in Equation \ref{functionalrep}, then
\begin{equation*}
\tilde{T}\left(\overline{\psi_{a, b}}\right)=T\left(\overline{\psi_{a, b}}\right)=\int_{\mathbb{R}^{d}} \overline{\psi_{a, b}}(x) d \mu(x)=\mathcal{W}_{\psi} \mu(a, b).
\end{equation*} Next, take $f= \overline{\psi_{a,b}}$ in Equation \ref{abc}, we get 
\begin{align*}
T\left(\overline{\psi_{a, b}}\right) & =\dfrac{1}{C_{\psi}}\int_{\mathbb{R}^{\times}} \int_{\mathbb{R}^{d}} g(r, s) \mathcal{W}_{\psi} \overline{\psi_{a, b}} (r, s) ds \frac{ d r}{|r|^{d+1}} \\
& =\dfrac{1}{C_{\psi}} \int_{\mathbb{R}^{\times}} \int_{\mathbb{R}^{d}} g(r, s)\left\langle \overline{\psi_{a, b}}, \psi_{r, s}\right\rangle_{L^{2}(\mathbb{R}^{d})}ds \frac{ d r}{|r|^{d+1}} \\
& =g(a, b).
\end{align*} The last equality follows from condition $(a)$. Therefore, $\mathcal{W}_{\psi} \mu =g$. This proves the sufficiency. \\

\noindent $(\implies)$  Suppose that $g=\mathcal{W}_{\psi} \mu$ for some $\mu \in \mathcal{M}({\mathbb{R}^{d}})$. We have to show that $g$ satisfies conditions $(a)$ and $(b)$. Choose $f= \overline{\psi_{a,b}}$ in Proposition \ref{measure re} , to get condition $(a)$,   
\begin{equation*}
  \dfrac{1}{C_{\psi}}   \int_{\mathbb{R}^{\times}} \int_{\mathbb{R}^{d}} \mathcal{W}_{\psi} \mu(r,s) {\langle \overline{\psi_{a,b}}, \psi_{r,s} \rangle}_{L^{2}(\mathbb{R}^{d})} ds\dfrac{ dr }{|r|^{d+1}} = \int_{\mathbb{R}^{d}} \overline{\psi_{a,b}} (x) d\mu(x)= \mathcal{W}_{\psi}\mu(a,b)= g(a,b)
\end{equation*} for all $(a,b)\in \mathbb{R}^{\times} \times \mathbb{R}^{d}$. Again, by Proposition \ref{measure re}, for all $f\in C_{c}^{d+1}(\mathbb{R}^{d})$
\begin{align*}
    \left|\int_{\mathbb{R}^{\times}} \int_{\mathbb{R}^{d}} g(a, b) \mathcal{W}_{\psi} f(a, b) db \frac{ d a}{|a|^{d+1}}\right|&= \left|\int_{\mathbb{R}^{\times}} \int_{\mathbb{R}^{d}} g(a, b) {\left\langle f, \psi_{a,b}\right\rangle}_{L^{2}(\mathbb{R}^{d})} db \frac{d a}{|a|^{d+1}}\right|\\
    &= C_{\psi}\left|\int_{\mathbb{R}^{d}} f(x) d\mu(x)\right|\\
    &  \leq C_{\psi}\int_{\mathbb{R}^{d}} \left|f(x)\right| |d\mu(x)|\\
    &\leq C_{\psi}\|f\|_{\infty} \|\mu\|_{TV}.
\end{align*}  Choosing $C= C_{\psi}\|\mu\|_{TV}>0$ proves condition $(b)$. This completes the proof.
\end{proof}

The vanishing moment condition imposed on $\psi$ is necessary for the conclusion of integrability in Lemma \ref{main:lemma1}.  We provide a counterexample in dimension $1$, i.e., $\mathbb{R}$ with the Lebesgue measure. More precisely, we show that if 
\begin{equation*}
  \int_{\mathbb{R}} y \psi(y) dy =0
\end{equation*}
is violated, then there exists $f\in C_c^2(\mathbb R)$ such that
$$
\int_{0<a\leq 1}\int_{\mathbb R}
|a|^{-1/2}|W_\psi f(a,b)|\,db\,\frac{da}{a^2}
$$ is divergent.
\begin{example}\label{eg1}
Let
$$
\varphi(x)=
\begin{cases}
C.e^{-\frac{1}{1-x^2}},& |x|<1,\\
0,& |x|\geq 1.
\end{cases}
$$
Then $\varphi\in C_c^\infty(\mathbb R)$ and choose the constant $C$ such that 
$$
\int_{\mathbb R}\varphi(y)\,dy=1 >0.$$

We next show that the first vanishing moment is also necessary, even when the zeroth moment vanishes.
With the same function $\varphi$ as above, put
\[
\psi=\varphi'.
\]
Then $\psi\in C_c^\infty(\mathbb R)$, and, since $\varphi(\pm1)=0$,
\[
\int_{\mathbb R}\psi(y)\,dy
=
\int_{-1}^{1}\varphi'(y)\,dy
=
\varphi(1)-\varphi(-1)
=0.
\]
On the other hand, integration by parts gives
\[
\int_{\mathbb R}y\psi(y)\,dy
=
\int_{-1}^{1}y\varphi'(y)\,dy
=
\big[y\varphi(y)\big]_{-1}^{1}
-\int_{-1}^{1}\varphi(y)\,dy
=-1\neq0.
\]
Choose $f\in C_c^2(\mathbb R)$ and a compact interval
$J$ such that
\[
|f'(b)|\geq c>0,\qquad \forall b\in J.
\]
 Taylor's theorem gives
 \[
f(b+ay)
=
f(b)+ayf'(b)
+\frac{a^2y^2}{2}f''(b+\theta ay),
\]
and consequently
\[
\begin{aligned}
W_\psi f(a,b)
&=
a^{\frac{1}{2}}\int_{\mathbb R}
f(b+ay)\psi(y)\,dy\\
&=
a^{\frac{1}{2}}f(b)\int_{\mathbb R}\psi(y)\,dy
+a^{\frac{3}{2}}f'(b)\int_{\mathbb R}y\psi(y)\,dy\\
&\qquad
+\frac{a^{\frac{5}{2}}}{2}
\int_{\mathbb R}y^2\psi(y)
f''(b+\theta ay)\,dy.
\end{aligned}
\]
Since $\displaystyle \int_{\mathbb R}\psi=0$ and
$\displaystyle \int_{\mathbb R}y\psi=-1$,
\[
W_\psi f(a,b)
=
-a^{\frac{3}{2}}f'(b)
+
R(a,b),
\]
where
\[
R(a,b)
=
\frac{a^{\frac{5}{2}}}{2}
\int_{\mathbb R}y^2\psi(y)
f''(b+\theta ay)\,dy.
\]
Moreover,
\[
|R(a,b)|
\leq
\frac{a^{\frac{5}{2}}}{2}
\|f''\|_\infty
\int_{\mathbb R}y^2|\psi(y)|\,dy.
\]
Set
\[
C_0=
\frac12\|f''\|_\infty
\int_{\mathbb R}y^2|\psi(y)|\,dy.
\]
Then
\[
|R(a,b)|\leq C_0a^{\frac{5}{2}}.
\]
Since $|f'(b)|\geq c$ on $J$, we obtain using reverse triangle inequality that
\[
|W_\psi f(a,b)|
\geq
c\,a^{\frac{3}{2}}-C_0a^{\frac{5}{2}}
=
a^{\frac{3}{2}}(c-C_0a).
\]
Choose $0<a_1\leq1$ such that
\[
C_0a_1\leq\frac{c}{2}.
\]
Then
\[
|W_\psi f(a,b)|
\geq
\frac{c}{2}a^{\frac{3}{2}},
\qquad \forall b\in J,\quad 0<a\leq a_1.
\]
Therefore
\[
\begin{aligned}
\int_{0<a\leq a_1}\int_{\mathbb R}
a^{-\frac{1}{2}}|W_\psi f(a,b)|\,db\,\frac{da}{a^2}
&\geq
\int_0^{a_1}\int_J
a^{-\frac{1}{2}}\frac{c}{2}a^{\frac{3}{2}}\,db\,\frac{da}{a^2}\\
&=
\frac{c m(J) }{2}
\int_0^{a_1}\frac{da}{a} \quad (\mbox{Note that $m(J)$ is the Lebesgue measure of $J$ })\\
&=\infty.
\end{aligned}
\]
Hence
\[
\int_{\mathbb R}y\psi(y)\,dy=0
\]
is also necessary.
\end{example}

\subsection{Consequences of the Main Result}\label{conseq} This section discusses some consequences of Theorem \ref{main:thm1} and provide examples to illustrate them.\\

The following result gives a concrete realization of $\mathcal{M}(\mathbb{R}^{d})$ via the wavelet transform: 
\begin{theorem}\label{main;thm 2}
    Let $\psi$ denote a wavelet satisfying the Assumption \ref{assumptions}. Then $\left(\mathcal{F}_{\psi}, \|\cdot \|_{\psi}\right)$ is a Banach space isometrically isomorphic to $\mathcal{M}(\mathbb{R}^{d}).$
\end{theorem}
\begin{proof} It follows from Proposition \ref{image} that the space $\mathcal{F}_{\psi}$ is a normed space.
    By Theorem \ref{main:thm1}  
        \begin{equation*}
            \mathcal{W}_{\psi} \left(\mathcal{M}(\mathbb{R}^{d})\right)= \mathcal{F}_{\psi}
        \end{equation*}  for any $\psi$ satisfying the regularity assumptions given in Theorem \ref{main:thm1}. Thus, $\mathcal{W}_{\psi}$ is surjective. Next, we prove that the map $\mathcal{W}_{\psi} : \mathcal{M}(\mathbb{R}^{d})\to \mathcal{F}_{\psi}$ is injective. Note that $\mathcal{W}_{\psi}$ is linear. Suppose $\mathcal{W}_{\psi} \mu=0$. Then from Equation \ref{measurereconstruction}, we have the bounded linear functional
        \begin{equation*}
           T(f)= \int_{\mathbb{R}^{d}} f(x) d\mu(x)=0
        \end{equation*} for all $f\in C_{c}^{d+1}(\mathbb{R}^{d})$. The density of $C_{c}^{d+1}(\mathbb{R}^{d})$ in $C_{0}(\mathbb{R}^{d})$ implies that $T(f)=0$ for all $f\in C_{0}(\mathbb{R}^{d})$. The uniqueness in the Riesz representation theorem implies that $\mu=0$.  Therefore, $\mathcal{W}_{\psi}: \mathcal{M}(\mathbb{R}^{d})\to \mathcal{F}_{\psi}$ is injective. Thus, it is a linear isomorphism from $\mathcal{M}(\mathbb{R}^{d})$ to $\mathcal{F}_{\psi}$. We claim that the map $\mathcal{W}_{\psi}: \mathcal{M}(\mathbb{R}^{d})\to \mathcal{F}_{\psi}$ is a linear isometry. By Proposition \ref{measure re}, we have for $\mu\in \mathcal{M}(\mathbb{R}^{d})$, 
        \begin{align*}
            \|\mathcal{W}_{\psi} \mu\|_{\psi}&=\sup \left\{\dfrac{1}{C_{\psi}}\left| \int_{\mathbb{R}^{\times}} \int_{\mathbb{R}^{d}} \mathcal{W}_{\psi}\mu(a,b) \mathcal{W}_{\psi} f(a,b) db \dfrac{da}{|a|^{d+1}}\right|: f\in C_{c}^{d+1}(\mathbb{R}^{d}), \quad \|f\|_{\infty}\leq 1\right\}\\
            &= \sup \left\{\left| \int_{\mathbb{R}^{d}} f(x) d\mu(x)\right|: f\in C_{c}^{d+1}(\mathbb{R}^{d}), \quad \|f\|_{\infty}\leq 1\right\}.
        \end{align*} It follows from Riesz representation theorem  (see \cite[Theorem 7.17]{folland1999real}) that 
        \begin{equation*}
            \|\mu\|_{TV}= \sup\left\{ \left| \int_{\mathbb{R}^{d}} f(x) d\mu(x)\right|: f\in C_{0}(\mathbb{R}^{d}), \quad \|f\|_{\infty} \leq 1\right\}.
        \end{equation*} Since $C_{c}^{d+1}(\mathbb{R}^{d})\subset C_{0}(\mathbb{R}^{d})$, then $\|\mathcal{W}_{\psi} \mu\|_{\psi} \leq \|\mu\|_{TV}$. 
        Next, we prove the reverse inequality. Let $\epsilon > 0$. There exists $f\in C_{0}(\mathbb{R}^{d})$ with $\|f\|_{\infty}\leq 1$ such that 
        \begin{equation}
            \|\mu\|_{TV}- \epsilon < \left|\int_{\mathbb{R}^{d}} f(x) d\mu(x)\right|. \label{inequality}
        \end{equation} The density of $C_{c}^{d+1}(\mathbb{R}^{d})$ in $C_{0}(\mathbb{R}^{d})$ implies that there exists $h\in C_{c}^{d+1}(\mathbb{R}^{d})$ such that 
        \begin{equation*}
            \|f-h\|_{\infty} < \epsilon.
        \end{equation*} As a consequence, 
        \begin{equation*}
            \|h\|_{\infty}\leq \| h-f\|_{\infty}+ \|f\|_{\infty}\leq 1+ \epsilon.
        \end{equation*} Define $\displaystyle g = \dfrac{h}{1+\epsilon}$. Then $g\in C_{c}^{d+1}(\mathbb{R}^{d})$ and $\|g\|_{\infty}\leq 1$. Now 
        \begin{align*}
            \|f-g\|_{\infty}&= \left\| f- \dfrac{h}{1+\epsilon}\right\|_{\infty}\\
            &\leq  \left\| f- \dfrac{f}{1+\epsilon}\right\|_{\infty}+ \left\| \dfrac{f}{1+ \epsilon}- \dfrac{h}{1+\epsilon}\right\|_{\infty}\\
            &= \|f\|_{\infty} \left(1 - \dfrac{1}{1+ \epsilon}\right) + \dfrac{1}{1+\epsilon} \|f-h\|_{\infty}\\
            &\leq \dfrac{2\epsilon}{1+ \epsilon}.
        \end{align*} Therefore, 
        \begin{align*}
            \left|\int_{\mathbb{R}^{d}} f(x) d\mu(x)\right|
            &\leq \left|\int_{\mathbb{R}^{d}} [f(x)-g(x)] d\mu(x)\right|+ \left|\int_{\mathbb{R}^{d}} g(x) d\mu(x)\right|\\
            &\leq \dfrac{2 \epsilon}{ 1+ \epsilon}\|\mu\|_{TV}+ \|\mathcal{W}_{\psi} \mu\|_{\psi}.
        \end{align*} Thus simplifying using the Equation \ref{inequality}, we get
        \begin{equation*}
            \|\mu\|_{TV} <   \left(\dfrac{1+\epsilon}{ 1- \epsilon}\right)\epsilon + \left(\dfrac{1+\epsilon}{ 1- \epsilon}\right)\|\mathcal{W}_{\psi} \mu\|_{\psi}.
        \end{equation*} Letting $\epsilon\to 0$, we get $\|\mu\|_{TV} \leq \|\mathcal{W}_{\psi} \mu\|_{\psi}.$ This proves that $\|\mathcal{W}_{\psi} \mu\|_{\psi}= \|\mu\|_{TV}.$ We have shown that $\mathcal{W}_{\psi}$ is an isometric linear isomorphism between $\mathcal{M}(\mathbb{R}^{d})$ and $\mathcal{F}_{\psi}$. Since $\left(\mathcal{M}(\mathbb{R}^{d}), \|\cdot\|_{TV}\right)$ is a Banach space, then $\left(\mathcal{F}_{\psi},\| \cdot \|_{\psi}\right)$ is also a Banach space.  This completes the proof.
\end{proof}
We note an immediate consequence of Theorem \ref{main;thm 2}. 
\begin{corollary}
    Let $\mu, \nu\in \mathcal{M}(\mathbb{R}^{d})$. Define 
        \begin{equation}
            \mathcal{W}_{\psi} \mu \odot \mathcal{W}_{\psi} \nu:= \mathcal{W}_{\psi} \left( \mu * \nu\right).
        \end{equation} Then $(\mathcal{F}_{\psi},\|\cdot \|_{\psi}, \odot)$ is a commutative unital Banach algebra with unit $\mathcal{W}_{\psi}\delta_{\mathrm{0}}$.
\end{corollary}
\begin{proof}
    The proof is a direct application of Theorem \ref{main;thm 2}. The isometry property of $\mathcal{W}_{\psi}$ together with the Banach algebra structure of $\mathcal{M}(\mathbb{R}^{d})$ under convolution $*$ and unit $\delta_0$ proves that $\mathcal{F}_{\psi}$ is a commutative unital Banach algebra under the operation $\odot.$ This completes the proof.
\end{proof}

\begin{remark}
  \begin{enumerate}
      \item Let $\psi$ be a wavelet satisfying the Assumption \ref{assumptions}. Then $\mathcal{W}_{\psi} \mu$ is continuous on $\mathbb{R}^{\times} \times \mathbb{R}^{d}.$ For, let $(a,b)\in \mathbb{R}^{\times}\times \mathbb{R}^{d}$ and $\left((a_n, b_n)\right)_{n\in \mathbb{N}} \subset \mathbb{R}^{\times} \times \mathbb{R}^{d}$ denotes a sequence converging to $(a,b)$ in the Euclidean topology. By bounded convergence theorem, $\mathcal{W}_{\psi} \mu(a_n, b_n)\to \mathcal{W}_{\psi} \mu(a, b)$ as $n\to \infty.$ That is, $\mathcal{F}_{\psi} \subset C\left(\mathbb{R}^{\times} \times \mathbb{R}^{d}\right).$
    \item Note that the continuous wavelet transform of $L^{2}-$functions are uniformly bounded. In the measure theoretic setting, the wavelet transform of a complex Radon measure may not be even  pointwise bounded. For example, let $\mu= \delta_{x_{\mathrm{0}}}$ for $x_{\mathrm{0}}\in \mathbb{R}$. Then 
        \begin{equation*}
            \mathcal{W}_{\psi} \mu(a,b) = \int_{\mathbb{R}} \dfrac{1}{\sqrt{|a|}} \overline{\psi\left(\dfrac{x-b}{a} \right)}d\delta_{x_{\mathrm{0}}}(x)= \dfrac{1}{\sqrt{|a|}} \overline{\psi\left(\dfrac{x_{\mathrm{0}}-b}{a} \right)}
        \end{equation*} for all points $(a,b)\in \mathbb{R}^{\times} \times \mathbb{R}.$ As $|a|\to 0$, $\mathcal{W}_{\psi} \mu(a,b)$ blows up, and hence is not bounded.
  \end{enumerate} 
\end{remark}

Next, we discuss some integrability properties of functions in $\mathcal{F}_{\psi}$. In particular, we prove weighted integrability of functions in $\mathcal{F}_{\psi}$ with respect to a suitable choice of family of weights independent of the wavelet $\psi$. 

\begin{proposition}\label{main;thm3}
   Let $1\leq p < \infty$ and let $w_{p}: \mathbb{R}^{\times}\to (0,\infty)$ denote the weight defined by 
    \begin{equation*}
        w_{p}(a)=\begin{cases}
            |a|^{\alpha}, & \quad 0<|a|\leq 1\\
            1, & \quad |a|>1
        \end{cases},
    \end{equation*} where $\alpha> \dfrac{dp}{2}$. Then $\mathcal{W}_{\psi} \mu \in L^{p}_{w_{p}}\left(\mathbb{R}^{\times}\times \mathbb{R}^{d} , db \dfrac{da}{|a|^{d+1}}\right)$ for any $\mu\in \mathcal{M}(\mathbb{R}^{d})$. In particular, 
    \begin{equation*}
        \|\mathcal{W}_{\psi}\mu\|_{p,w_{p}}\leq C \|\mu\|_{TV}
    \end{equation*} for some constant $C>0$.
\end{proposition}
\begin{proof}
    Let $p\in [1,\infty)$ and $q$ denote the conjugate exponent of $p$. Note that $1\in L^{q} \left(\mathbb{R}^{d}, |\mu|\right)$. We need to show that 
    \begin{equation*}
        \int_{\mathbb{R}^{\times}} \int_{\mathbb{R}^{d}}  \left|\mathcal{W}_{\psi}\mu(a,b)\right|^{p} w_{p}(a) db \dfrac{da}{|a|^{d+1}}< \infty. 
    \end{equation*} For, observe that 
    \begin{align*}
        \int_{\mathbb{R}^{d}}  \left|\mathcal{W}_{\psi}\mu(a,b)\right|^{p} db&= \int_{\mathbb{R}^{d}} \left|\int_{\mathbb{R}^{d}} |a|^{-\frac{d}{2}} \overline{\psi \left(\dfrac{x-b}{a} \right)} d\mu(x)\right|^{p} db\\
        &\leq |a|^{-\frac{dp}{2}}\int_{\mathbb{R}^{d}} \left\{\int_{\mathbb{R}^{d}} \left|\psi \left(\dfrac{x-b}{a} \right)\right| d|\mu|(x) \right\}^{p} db\\
        \mbox{By using H\"older's inequality}\\ 
        &\leq |a|^{-\frac{dp}{2}}\int_{\mathbb{R}^{d}} \left\{ \left(\int_{\mathbb{R}^{d}} \left|\psi \left(\dfrac{x-b}{a} \right)\right|^{p} d|\mu|(x)\right)^{\frac{1}{p}} \left(\int_{\mathbb{R}^{d}}1\cdot d|\mu|(x)\right)^{\frac{1}{q}}\right\}^{p} db \\
        &= |a|^{-\frac{dp}{2}} \|\mu\|_{TV}^{\frac{p}{q}}\int_{\mathbb{R}^{d}} \int_{\mathbb{R}^{d}} \left|\psi \left(\dfrac{x-b}{a} \right)\right|^{p} d|\mu|(x) db\\
        &=  |a|^{-\frac{dp}{2}} \|\mu\|_{TV}^{\frac{p}{q}}\int_{\mathbb{R}^{d}} \left[\int_{\mathbb{R}^{d}} \left|\psi \left(\dfrac{x-b}{a} \right)\right|^{p}  db\right] d|\mu|(x)  \left(\mbox{ Fubini-Tonelli theorem}\right)\\
        \mbox{Substitute $u= \dfrac{x-b}{a}$ we get,}\\
        &=|a|^{d \left(1-\frac{p}{2}\right)} \|\mu\|_{TV}^{\frac{p}{q}} \int_{\mathbb{R}^{d}} \left[\int_{\mathbb{R}^{d}} \left|\psi(u)\right|^{p} du\right] d|\mu|(x)\\
        &=|a|^{d \left(1-\frac{p}{2}\right)} \|\mu\|_{TV}^{\frac{p}{q}+1} \|\psi\|_{p}^{p}\\
        &= |a|^{d \left(1-\frac{p}{2}\right)} \|\mu\|_{TV}^{p} \|\psi\|_{p}^{p} \quad \left(\mbox{Since $\dfrac{1}{p}+ \dfrac{1}{q}=1$}\right).
    \end{align*} Thus, we proved the estimate
    \begin{equation}\label{es1}
        \int_{\mathbb{R}^{d}}  \left|\mathcal{W}_{\psi}\mu(a,b)\right|^{p} db \leq  |a|^{d \left(1-\frac{p}{2}\right)} \|\mu\|_{TV}^{p} \|\psi\|_{p}^{p}.
    \end{equation} Integrating Equation \ref{es1} over $\mathbb{R}^{\times}$ with respect to $w_{p}(a) \dfrac{da}{|a|^{d+1}}$ yields that
    \begin{align*}
         \int_{\mathbb{R}^{\times}} w_{p}(a)\int_{\mathbb{R}^{d}}  \left|\mathcal{W}_{\psi}\mu(a,b)\right|^{p} db \dfrac{da}{|a|^{d+1}} &\leq  \|\mu\|_{TV}^{p} \|\psi\|_{p}^{p} \int_{\mathbb{R}^{\times}} w_{p}(a) |a|^{d \left(1-\frac{p}{2}\right)} \dfrac{da}{|a|^{d+1}}\\
         &= \|\mu\|_{TV}^{p} \|\psi\|_{p}^{p} \int_{\mathbb{R}^{\times}} w_{p}(a) |a|^{-1-\frac{dp}{2}} da\\
         &= \|\mu\|_{TV}^{p} \|\psi\|_{p}^{p} \left[\int_{0<|a|\leq 1}|a|^{\left(\alpha-\frac{dp}{2}\right)-1} da +\int_{|a|>1} |a|^{-1-\frac{dp}{2}} da \right].
    \end{align*} The integral $\displaystyle \int_{0<|a|\leq 1}|a|^{\left(\alpha-\frac{dp}{2}\right)-1} da <\infty$ because $\alpha > \dfrac{dp}{2}$, and $\displaystyle \int_{|a|>1} |a|^{-1-\frac{dp}{2}} da $ is finite always. This then shows that $\mathcal{W}_{\psi} \mu \in L^{p}_{w_p}\left(\mathbb{R}^{\times}\times \mathbb{R}^{d} , db \dfrac{da}{|a|^{d+1}}\right)$ and 
    \begin{equation*}
    \|\mathcal{W}_{\psi}\mu\|_{p,w_p}\leq C \|\mu\|_{TV}
    \end{equation*} where $C>0$ denote a suitable constant. This completes the proof.
\end{proof}
 This then shows that convergence of $\{\mu_n\}_n$ to $\mu$ in $\mathcal{M}(\mathbb{R}^d)$ in the total variation norm will imply convergence of the wavelet transforms $\mathcal{W}_{\psi}(\mu_n)$ to $\mathcal{W}_{\psi}(\mu)$ in the weighted $L^p_{w_p}(\mathbb{R^{\times}}\times\mathbb{R}^d)$ space for a specific family of weights on $\mathbb{R}^{\times}$.

In the above proof, Equation \ref{es1} gives us a family of bounded linear operators from $\mathcal{M}(\mathbb{R}^d)$ to $L^p(\mathbb{R}^d)$, which we formalize in the following corollary.

\begin{corollary}
 Let $\psi$ denote a wavelet satisfying the Assumption \ref{assumptions}. For every $a\in \mathbb{R}^{\times}$, define $\mathcal{W}_{\psi, a}: \mathcal{M}(\mathbb{R}^{d})\to L^{p} (\mathbb{R}^{d})$, where $\mathcal{W}_{\psi, a}(\mu)= \mathcal{W}_{\psi}\mu (a,\cdot)$. Then $\mathcal{W}_{\psi, a}$ is a bounded linear operator, and
    \begin{equation*}
        \|\mathcal{W}_{\psi} \mu(a,\cdot)\|_p \leq |a|^{\frac{d}{p} \left(1-\frac{p}{2}\right)} \|\mu\|_{TV}\|\psi\|_{p}
    \end{equation*} for $1\leq p<\infty$.    
\end{corollary}

Next, given the wavelet transform of a measure, we use our results to compute the measure on closed rectangles. First we introduce the following notations; Let $\phi: \mathbb{R}\to \mathbb{R}$ be the function given by 
\begin{equation*}
    \phi(x)= \begin{cases}
             e^{-\frac{1}{1-x^{2}}}, & |x|< 1\\
            0,& |x|\geq 1
        \end{cases}
\end{equation*} for all $x\in \mathbb{R}$. Note that $\phi\in C_{c}^{\infty}(\mathbb{R}).$ Define $\Phi: \mathbb{R}^{d} \to \mathbb{R}$ by 
\begin{equation*}
    \Phi(x_1, x_2, \cdots, x_{d})= C \prod_{i=1}^{d} \phi(x_i),
\end{equation*} where $(x_1, x_2, \cdots, x_d)\in \mathbb{R}^{d},$ and $C$ is a constant such that $\displaystyle \int_{\mathbb{R}^{d}} \Phi(x) dx=1.$ Clearly, $\Phi \in C_{c}^{\infty}(\mathbb{R}^{d}).$ For $\epsilon >0$, let $\displaystyle \Phi_{\epsilon}(x)=\displaystyle  \dfrac{1}{\epsilon^{d}}\Phi\left(\dfrac{x}{\epsilon}\right)$. Define $\mathcal{R}_{\epsilon}=\displaystyle  \prod_{i=1}^{d} \left[a_{i}-\epsilon, b_{i} + \epsilon \right]$ and let 
\begin{equation}\label{molli}
    f_{\epsilon}(x)= \left(\chi_{\mathcal{R}_{\epsilon}} * \Phi_{\epsilon}\right)(x)
\end{equation} for all $x\in \mathbb{R}^{d}.$ Note that $f_{\epsilon}\in C_{c}^{\infty}(\mathbb{R}^{d})$ and hence right hand side of equation \ref{measurereconstruction} is valid.\\

The following proposition provides a strategy to reconstruct the measure associated to a function $g\in \mathcal{F}_{\psi}$ from the continuous wavelet transforms of functions:

\begin{proposition}\label{proplimit}
    Let $g\in \mathcal{F}_{\psi}$ and $\mu$ denote the associated complex Radon measure.  Then for any closed rectangle $\mathcal{R}\subset \mathbb{R}^d$, 
    \begin{equation}\label{limiting}
            \mu(\mathcal{R})= \lim_{\epsilon \to 0}  \dfrac{1}{C_{\psi}}\int_{\mathbb{R}^{\times }} \int_{\mathbb{R}^{d}} g(r,s) \mathcal{W}_{\psi} f_{\epsilon}(r,s) ds \dfrac{ dr}{|r|^{d+1}},
        \end{equation} where $f_{\epsilon}$ is as in Equation \ref{molli}.
\end{proposition}
\begin{proof}
    Let $g\in \mathcal{F}_{\psi}$. By Theorem $\ref{main:thm1}$, there exists $\mu\in \mathcal{M}(\mathbb{R}^{d})$ such that $g= \mathcal{W}_{\psi}\mu$. We denote $\displaystyle \mathcal{R}= \prod_{i=1}^{d} [a_{i}, b_{i}]$, where $a_i, b_i \in \mathbb{R}$ with $a_{i} < b_{i}, i= 1,2, \cdots, d.$ Our objective is to compute $\mu(\mathcal{R})$ explicitly. Our first claim is that $f_{\epsilon}$ converges to $\chi_{\mathcal{R}}$ pointwise for all $x\in \mathbb{R}^{d}.$ Note that 
\begin{align*}
    f_{\epsilon}(x)& = \int_{\mathbb{R}^{d}} \chi_{\mathcal{R}_{\epsilon}}(x-y) \Phi_{\epsilon}(y) dy\\
    &= \int_{\mathbb{R}^{d}} \chi_{\mathcal{R}_{\epsilon}}(y) \Phi_{\epsilon}(x-y) dy\\
    &= \int_{\mathcal{R}_{\epsilon}} \Phi_{\epsilon}(x-y) dy.
\end{align*} 
Note that supp$(\phi)\subset [-1,1]$, and hence supp$(\Phi_{\epsilon})\subset \displaystyle \left[-\epsilon, \epsilon\right]^{d}$. Let $y=(y_1, y_2, \cdots, y_d)\in \mathbb{R}^{d}$.
It is clear that for a fixed $x=(x_1, x_2, \cdots, x_d)\in \mathbb{R}^{d}$,  $\Phi_{\epsilon}(x-y) $ is non-zero if and only if $x-y \in  \displaystyle \left[-\epsilon, \epsilon\right]^{d}$ if and only if $|x_{i}- y_{i}| < \epsilon$ for all $i=1,2,\cdots,d.$ 
Let $x=(x_1, x_2, \cdots, x_d)\in \mathcal{R}$. Denote $\mathcal{R}_{\epsilon}^{x}= \displaystyle \prod_{i=1}^{d} \left[x_{i}- \epsilon, x_{i} + \epsilon\right]$. Then $\Phi_{\epsilon}(x-y) $ is non-zero if and only if $y\in \mathcal{R}_{\epsilon}^{x} \subset \mathcal{R}_{\epsilon}.$ Therefore, $\displaystyle f_{\epsilon}(x)=\int_{\mathbb{R}^{d}} \Phi_{\epsilon}(x-y) dy=1$ for all $x\in  \mathcal{R}.$ Suppose that $x=(x_1,x_2,\cdots, x_d)\notin \mathcal{R}$. There exists an $i=1,2,\cdots,d$ such that $x_{i}\notin [a_{i}, b_{i}]$. In this case, we show that for sufficiently small $\epsilon>0$, if $y\in \mathcal{R}_{\epsilon}$, then $x- y\notin [-\epsilon,\epsilon]^{d}$. Assume $x_{i} < a_{i}$. That is, $a_i- x_i >0$. Choose a sufficiently small $\epsilon >0$ such that $0< \epsilon < \dfrac{a_i-x_i}{2}$. Note that $y=(y_1,y_2, \cdots, y_d)\in \mathcal{R}_{\epsilon}$ if and only if $y_{i}\in [a_i- \epsilon,b_i+\epsilon]$ for all $i=1,2,\cdots, d$. For $y_i \in [a_i- \epsilon,b_i+\epsilon] $, we have 
\begin{equation*}
    y_{i}- x_{i} \geq  a_{i}-\epsilon-x_{i} = (a_i- x_i)-\epsilon > (a_i- x_i)- \dfrac{a_i- x_i}{2}=\dfrac{a_i- x_i}{2} > \epsilon.
\end{equation*}  Similar argument can be used to show that the same thing happens when $x_{i}> b_{i}$ for some $i=1,2,\cdots, d$. This shows that $\Phi_{\epsilon}(x-y)=0$ for all $y\in \mathcal{R}_{\epsilon}$ which gives $f_{\epsilon}(x)=0$ for $x=(x_1,x_2,\cdots, x_d)\notin \mathcal{R}$. Thus, $\displaystyle \lim_{\epsilon\to 0} f_{\epsilon}(x)=0$ for all $x\notin \mathcal{R}.$ We have shown that $\displaystyle \lim_{\epsilon\to 0} f_{\epsilon}(x) = \chi_{\mathcal{R}}(x)$ pointwise for all $x\in \mathbb{R}^{d}.$ Observe that for all $x\in \mathbb{R}^{d}$
\begin{equation*}
    |f_{\epsilon}(x)|\leq \int_{\mathcal{R}_{\epsilon}} \Phi(y) dy \leq \int_{\mathbb{R}^{d}} \Phi(y) dy =1,
\end{equation*} and $1\in L^{1}(\mu).$ By dominated convergence theorem, 
\begin{equation*}
    \mu(\mathcal{R})= \int_{\mathbb{R}^{d}} \chi_{\mathcal{R}}(x) d\mu(x)=\int_{\mathbb{R}^{d}} \lim_{\epsilon\to 0} f_{\epsilon}(x) d\mu(x) = \lim_{\epsilon\to 0} \int_{\mathbb{R}^{d}} f_{\epsilon}(x) d\mu(x).
\end{equation*} Since $f_{\epsilon}\in C_{c}^{\infty}(\mathbb{R}^{d})$, then by Proposition \ref{measure re} we have 
\begin{equation*}
    \mu(\mathcal{R})=\lim_{\epsilon \to 0}  \dfrac{1}{C_{\psi}}\int_{\mathbb{R}^{\times }} \int_{\mathbb{R}^{d}} g(r,s) \mathcal{W}_{\psi} f_{\epsilon}(r,s) ds \dfrac{dr}{|r|^{d+1}}. 
\end{equation*} This completes the proof.
\end{proof}
The following example illustrates the use of Proposition  \ref{proplimit}.
\begin{example}\label{eg2}
    Choose a suitable $\psi$ satisfying the hypothesis of Theorem \ref{main:thm1}. Let $g(r,s)= \dfrac{1}{|r|^{\frac{1}{2}}} \overline{\psi \left(-\dfrac{s}{r}\right)}$ for all $(r,s)\in \mathbb{R}^{\times} \times \mathbb{R}.$ It is easy to see from the definition that $g= \mathcal{W}_{\psi} \delta_{\mathrm{0}}$ and hence belongs to $\mathcal{F}_{\psi}.$
Taking $\psi^1 = \psi$ and $\psi^2=\overline{\psi}$, in Lemma \ref{reconstruction lemma} it follows that 
    \begin{align*}
     \dfrac{1}{C_{\psi}}   \int_{\mathbb{R}^{\times}} \int_{\mathbb{R}} \overline{\psi_{r,s}(0)} \mathcal{W}_{\psi} \overline{\psi_{x,y}} (r,s) ds\dfrac{dr}{r^{2}}= \overline{\psi_{x,y}(0)}
    \end{align*} and 
    \begin{equation*}
      \dfrac{1}{C_{\psi}}  \left|\int_{\mathbb{R}^{\times}} \int_{\mathbb{R}} \overline{\psi_{r,s}(0)} \mathcal{W}_{\psi} f(r,s) ds \dfrac{ dr}{r^{2}} \right|= |f(0)|\leq 1\cdot \|f\|_{\infty}
    \end{equation*} for all $f\in C_{c}^{2}(\mathbb{R})$. Thus $g(r,s)= \dfrac{1}{|r|^{\frac{1}{2}}} \overline{\psi \left(-\dfrac{s}{r}\right)}$ satisfies the conditions $(a)$ and $(b)$ in Theorem \ref{main:thm1}, which gives that $g\in \mathcal{F}_{\psi}$ without using $g= \mathcal{W}_{\psi} \delta_{\mathrm{0}}$. Using Proposition \ref{proplimit}, we have for any $a,b \in \mathbb{R}$ with $a<b$,
    \begin{align*}
        \mu([a,b])= \lim_{\epsilon \to 0} \dfrac{1}{C_{\psi}}\int_{\mathbb{R}^{\times}} \int_{\mathbb{R}} \overline{\psi_{r,s}(0)} \mathcal{W}_{\psi} f_{\epsilon}(r,s) ds \dfrac{ dr}{r^{2}}= \lim_{\epsilon\to 0}f_{\epsilon}(0)= \delta_{\mathrm{0}}([a,b]).
    \end{align*} This follows from Proposition \ref{measure re}, and the fact that $f_{\epsilon}\to \chi_{[a,b]}$ as $\epsilon \to 0$.
\end{example}

 Next, we give a reconstruction formula of the measure $\mu$ on closed rectangles $\mathcal{R}$ in $\mathbb{R}^d$ directly from the wavelet transform of its characteristic function $\mathcal{W}_{\psi}\chi_{\mathcal{R}}$. Since the characteristic function of a closed rectangle $\chi_{\mathcal{R}}$ is not in $C_c^{d+1}(\mathbb{R}^d)$, we cannot directly use Proposition \ref{measure re}. Specifically, we address the question when can we take the limit $\epsilon\to 0$ inside the classical wavelet transform $\mathcal{W}_{\psi}f_{\epsilon}$ in Proposition \ref{proplimit}. For that we need the following preparatory lemmas and definitions.
\begin{definition}
    Let $g\in \mathcal{F}_{\psi}$. For each $r\in \mathbb{R}^{\times}$, define $\tilde{g}_{r}:\mathbb{R}^{d}\to \mathbb{C}$ by 
    \begin{equation*}
        \tilde{g}_{r} (s)= g(r,s), \quad \forall s\in \mathbb{R}^{d}.
    \end{equation*}
\end{definition}
\begin{lemma}
    Let $g\in \mathcal{F}_{\psi}$. {Then the function $\tilde{g}_{r} \in L^{1}(\mathbb{R}^{d})$ for every $r\in \mathbb{R}^{\times}$}.
\end{lemma}
\begin{proof}
    By Theorem \ref{main:thm1}, $g= \mathcal{W}_{\psi}\mu$ for some $\mu\in \mathcal{M}(\mathbb{R}^{d}).$ We fix $r\in \mathbb{R}^{\times}.$ Then 
    \begin{align*}
        \int_{\mathbb{R}^{d}} {|\tilde{g}_{r} (s)|}ds =&\int_{\mathbb{R}^{d}} \left|\mathcal{W}_{\psi} \mu(r,s)\right|ds\\ &=\int_{\mathbb{R}^{d}} \left[\left||r|^{-\frac{d}{2}} \int_{\mathbb{R}^{d}} \overline{\psi \left(\dfrac{x-s}{r}\right)} d\mu(x) \right|\right] ds\\
        &\leq |r|^{-\frac{d}{2}}\int_{\mathbb{R}^{d}} \int_{\mathbb{R}^{d}} \left|\psi \left(\dfrac{x-s}{r}\right)\right| d|\mu|(x) ds\\
        &=|r|^{-\frac{d}{2}} \int_{\mathbb{R}^{d}} \left[\int_{\mathbb{R}^{d}} \left|\psi \left(\dfrac{x-s}{r}\right)\right| ds\right] d|\mu|(x) \quad \left(\mbox{By Fubini-Tonelli theorem}\right)\\
        &=|r|^{\frac{d}{2}}  \int_{\mathbb{R}^{d}} \left[\int_{\mathbb{R}^{d}} |\psi(u)| du\right] d|\mu|(x) \quad \left(\mbox{Take $u= \dfrac{x-s}{r}$}\right)\\
        &= |r|^{\frac{d}{2}} \|\psi\|_{1} \|\mu\|_{TV}.
    \end{align*} This then gives that {$\tilde{g}_{r}\in L^{1}(\mathbb{R}^{d})$}.
\end{proof}
\begin{definition}
    Let $g\in \mathcal{F}_{\psi}$. Define $G: \mathbb{R}^{\times}\to \mathbb{C}$ by 
    \begin{equation*}
       { G(r)= \|\tilde{g}_{r}\|_{1} \quad \forall r\in \mathbb{R}^{\times}.}
    \end{equation*}
\end{definition}
\begin{proposition}\label{passinglimit}
    Let $g\in \mathcal{F}_{\psi}$. If
    \begin{equation}\label{refined assumption}
       \int_{0<|r|\leq 1} |r|^{-\left(1+ \frac{d}{2}\right)} G(r) dr< \infty,
    \end{equation} then 
    \begin{equation}\label{direct reconstruction}
        \mu(\mathcal{R})=\dfrac{1}{C_{\psi}} \int_{\mathbb{R}^{\times}} \int_{\mathbb{R}^{d}} g(r,s) \mathcal{W}_{\psi} \chi_{\mathcal{R}}(r,s) ds \dfrac{ dr}{|r|^{d+1}}
    \end{equation} for any closed rectangle $\mathcal{R}\subset \mathbb{R}^d.$
\end{proposition}

\begin{proof}
    Let $\mathcal{R}=\displaystyle \prod_{i=1}^{d} [a_i,b_i]$. We first claim that supp$(f_{\epsilon})\subset K_{\epsilon}$, where $K_{\epsilon}= \displaystyle \prod_{i=1}^{d} [a_i -2 \epsilon, b_i+ 2 \epsilon]$. Here, $\mathcal{R}_{\epsilon}$ is the closed rectangle $\displaystyle \prod_{i=1}^{d} [a_i - \epsilon, b_i+  \epsilon]$ and $f_{\epsilon}= \Phi_{\epsilon}* \chi_{\mathcal{R}_{\epsilon}}$. Suppose $x\notin K_{\epsilon}$. Then there exists $i=1,2,\cdots,d$ such that $x_{i} \notin [a_i- 2\epsilon, b_{i}+ 2 \epsilon]$. Without loss of generality assume that $x_{i} < a_{i}- 2 \epsilon$. Then for $y_{i} \in [a_{i}-\epsilon, b_{i}+\epsilon]$, we have 
    \begin{equation*}
        |y_{i}- x_{i}|\geq a_{i}-\epsilon-x_{i} > a_{i}- \epsilon+ 2\epsilon- a_{i}=\epsilon.
    \end{equation*} Recall that $\Phi_{\epsilon}(x-y)\neq 0$ if and only if $|x_i- y_i|< \epsilon$ for all $i=1,2,\cdots,d$. Therefore, $\Phi_{\epsilon}(x-y)=0$ for all $x\notin K_{\epsilon}$ and $y\in \mathcal{R}_{\epsilon}$. Fix $\epsilon_{1}>0$ and choose sufficiently small $\epsilon>0$ such that $\epsilon<\epsilon_{1}$. Thus, supp$(f_{\epsilon})\subset K_{\epsilon}\subset K_{\epsilon_1}= \displaystyle \prod_{i=1}^{d} [a_{i}- 2\epsilon_1, b_{i}+ 2\epsilon_1 ]$.

    Since $\psi$ is compactly supported, then there exists $R>0$ such that supp$(\psi)\subset [-R,R]^{d}.$ Furthermore, we know that supp$(\Phi_{\epsilon})\subset [-\epsilon, \epsilon]^{d}$.
     Let $K= [-R,R]^{d}\bigcap K_{\epsilon_1}$. We need to find out the set of all $s=(s_{1},s_{2}, \cdots, s_{d})\in \mathbb{R}^{d}$ such that $\mathcal{W}_{\psi} f_{\epsilon}(r,s)\neq 0$ for all $r\in \mathbb{R}^{\times}.$\\

     Suppose  $\mathcal{W}_{\psi} f_{\epsilon}(r,s)\neq 0$. That is, $\displaystyle |r|^{-\frac{d}{2}} \int_{\mathbb{R}^{d}} f_{\epsilon}(x) \overline{\psi \left( \dfrac{x-s}{r}\right)} dx \neq 0.$ Then there exists at least one $x\in K.$ That is, $x\in K_{\epsilon_{1}}$ and $\dfrac{x-s}{r}\in [-R,R]^{d}$. Equivalently, for all $i=1,2,\cdots,d$
     \begin{align*}
         -R|r|\leq x_i- s_i \leq R|r|&\iff -x_{i}- R|r|\leq -s_{i}\leq R|r| -x_{i} \\
         &\iff x_{i} - R|r|\leq s_{i}\leq R|r| + x_{i}\\
         &\iff a_{i}- 2 \epsilon_{1}- R|r| \leq s_{i}\leq b_{i}+ 2\epsilon_{1} + R|r|. 
     \end{align*} Let $\mathcal{K}_{i}= [a_{i}- 2 \epsilon_{1}- R|r|, b_{i}+ 2\epsilon_{1} + R|r|]$. Define $\mathcal{K}= \displaystyle \prod_{i=1}^{d} \mathcal{K}_{i}$. Note that $\mathcal{K}$ is compact. Thus, we have shown that $\mathcal{W}_{\psi}f_{\epsilon}(r,s)\neq 0$ implies that $s\in \mathcal{K}$. As a consequence 
     \begin{equation}\label{decompmolli}
         \mathcal{W}_{\psi}f_{\epsilon}(r,s)= \mathcal{W}_{\psi}f_{\epsilon}(r,s) \chi_{\mathcal{K}}(s)
     \end{equation}
 \noindent We look at two different cases: $|r|\leq 1$ and $|r|>1$. For the both cases we denote  $F_{\epsilon}(r,s)= g(r,s) \mathcal{W}_{\psi} f_{\epsilon}(r,s)$.\\
     
 First, let us assume that $0<|r|\leq 1$. Observe that
\begin{align*}
    |F_{\epsilon}(r,s)|&= \left|g(r,s)\right|\left|\mathcal{W}_{\psi} f_{\epsilon}(r,s)\right|\\
    &= \left|g(r,s)\right| \left||r|^{-\frac{d}{2}}\int_{\mathbb{R}^{d}} f_{\epsilon}(x) \overline{\psi\left(\dfrac{x-s}{r}\right)} dx\right|\\
    &\leq |r|^{-\frac{d}{2}}  \left|g(r,s)\right| \int_{\mathbb{R}^{d}} \left|f_{\epsilon}(x)\right| \left|{\psi\left(\dfrac{x-s}{r}\right)}\right| dx\\
    &\leq |r|^{-\frac{d}{2}}  \left|g(r,s)\right| \int_{\mathbb{R}^{d}} \left|{\psi\left(\dfrac{x-s}{r}\right)}\right| dx\\
    &= |r|^{\frac{d}{2}}  \left|g(r,s)\right| \int_{\mathbb{R}^{d}} |\psi(y)| dy \quad \left(\mbox{Put $y= \dfrac{x-s}{r}$}\right)\\
    &= \|\psi\|_{1} |r|^{\frac{d}{2}}  \left|g(r,s)\right|.
\end{align*} Define $G(r,s)= |r|^{\frac{d}{2}}  \left|g(r,s)\right|$. Note that 
\begin{align*}
    \int_{0<|r|\leq 1} \int_{\mathbb{R}^{d}} |G(r,s)| ds \dfrac{dr}{|r|^{d+1}}&= \int_{0< |r|\leq 1} \int_{\mathbb{R}^{d}} |r|^{\frac{d}{2}}  \left|g(r,s)\right| ds \dfrac{dr}{|r|^{d+1}}\\
    &= \int_{0<|r|\leq 1} \int_{\mathbb{R}^{d}} |r|^{-\left(1+ \frac{d}{2}\right)}  \left|g(r,s)\right| ds dr\\
    &= \int_{0<|r|\leq 1} |r|^{-\left(1+ \frac{d}{2}\right)} \left[\int_{\mathbb{R}^{d}}\left|g(r,s)\right| ds \right]dr\\
    &= \int_{0<|r|\leq 1} |r|^{-\left(1+ \frac{d}{2}\right)} \|\tilde{g}(r,\cdot)\|_{1} dr\\
    &= \int_{0<|r|\leq 1} |r|^{-\left(1+ \frac{d}{2}\right)} G(r) dr\\
    &<\infty.
\end{align*} The last finiteness follows from Equation \ref{refined assumption}. This shows that $G(r,s)\in L^{1}\left(M \times \mathbb{R}^{d}, ds\dfrac{dr}{|r|^{d+1}}\right)$, where $M=\left\{r\in \mathbb{R}^{\times} : 0<|r|\leq 1\right\}$.
\\
   
    Next, assume that $|r|>1$. By Conditions $(a)$ and $(b)$ of Theorem \ref{main:thm1}, there exists $C_{6}>0$ such that 
    \begin{equation*}
        |g(r,s)|\leq C_{6} \|\psi_{r,s}\|_{\infty}\leq C_{6}|r|^{-\frac{d}{2}} \|\psi\|_{\infty}.
    \end{equation*} Set $C_{7}= C_{6} \|\psi\|_{\infty}.$ Furthermore, 
    \begin{align*}
        \left|\mathcal{W}_{\psi} f_{\epsilon}(r,s)\right|&\leq  |r|^{-\frac{d}{2}} \|\psi\|_{\infty} \|f_{\epsilon}\|_{1}.
    \end{align*} Observe that 
    \begin{align*}
        \int_{\mathbb{R}^{d}} |f_{\epsilon} (x)| dx= \int_{\mbox{supp}(f_{\epsilon})} |f_{\epsilon} (x)| dx \leq \int_{K_{\epsilon_{1}}} |f_{\epsilon} (x)| dx \leq m(K_{\epsilon_{1}}).
    \end{align*} Therefore, 
    \begin{equation*}
        |\mathcal{W}_{\psi} f_{\epsilon}(r,s)| \leq C_{8} |r|^{-\frac{d}{2}},
    \end{equation*} where $C_{8}= \|\psi\|_{\infty} m(K_{\epsilon_{1}}).$ By Equation \ref{decompmolli} together with preceeding estimates, 
    \begin{align*}
       | F_{\epsilon}(r,s)|& =\left|g(r,s) \mathcal{W}_{\psi} f_{\epsilon}(r,s)\right|\\&= \left|g(r,s) \mathcal{W}_{\psi} f_{\epsilon}(r,s)\right| \left|\chi_{\mathcal{K}}(s)\right|\\
        &\leq C_{9} |r|^{-d} \left|\chi_{\mathcal{K}}(s)\right|,
    \end{align*} where $C_{9}= C_{7} C_{8}$. Integrating with respect to $\dfrac{dr}{|r|^{d+1}}$ over the domain $\{a\in \mathbb{R}^{\times} : |a|>1\}$ yields that 
    \begin{align*}
        \int_{|r|>1} \int_{\mathbb{R}^{d}} |r|^{-d} \left|\chi_{\mathcal{K}}(s)\right| \dfrac{ds dr}{|r|^{d+1}}= \int_{|r|>1} |r|^{-d} m(\mathcal{K}) \dfrac{dr}{|r|^{d+1}}=\int_{|r|>1}  m(\mathcal{K}) \dfrac{dr}{|r|^{2d+1}} <\infty.
    \end{align*}  $m(\mathcal{K})$ is having maximum degree $|r|^d$ and hence with $|r|^{2d + 1}$ in the denominator, it is integrable. Hence, the finiteness in the right hand side is justified. 
    
     Note that $\mathcal{W}_{\psi} f_{\epsilon} (r,s)$ converges to $\mathcal{W}_{\psi} \chi_{\mathcal{R}}(r,s)$ for all $(r,s)\in \mathbb{R}^{\times} \times \mathbb{R}^{d}$. Indeed, for all $x\in \mathbb{R}^{d}$
    \begin{equation*}
        |r|^{-\frac{d}{2}}\left|f_{\epsilon}(x) \psi\left(\dfrac{x-s}{r}\right)\right|\leq |r|^{-\frac{d}{2}}  \left| \psi\left(\dfrac{x-s}{r}\right)\right| 
    \end{equation*} and 
    \begin{equation*}
       |r|^{-\frac{d}{2}} \lim_{\epsilon \to 0} f_{\epsilon}(x)  \psi\left(\dfrac{x-s}{r}\right)=  |r|^{-\frac{d}{2}} \chi_{\mathcal{R}}(x)\psi\left(\dfrac{x-s}{r}\right)
    \end{equation*} since $f_{\epsilon}(x)\to \chi_{\mathcal{R}}(x)$ as $\epsilon \to 0$. The use of dominated convergence theorem implies that $\mathcal{W}_{\psi} f_{\epsilon}(r,s)\to \mathcal{W}_{\psi} \chi_{\mathcal{R}}(r,s)$ as $\epsilon\to 0$ for all $(r,s)\in \mathbb{R}^{\times}\times \mathbb{R}^{d}.$ This then implies that $F_{\epsilon}(r,s)\to g(r,s)  \mathcal{W}_{\psi} \chi_{\mathcal{R}}(r,s)$ as $\epsilon \to 0$ for all $(r,s)\in \mathbb{R}^{\times} \times \mathbb{R}^{d}$.
    Therefore, applying dominated convergence theorem to Case 1 and Case 2, we get 
    \begin{equation*}
        \lim_{\epsilon\to 0}\dfrac{1}{C_{\psi}} \int_{0<|r|\leq 1} \int_{\mathbb{R}^{d}} g(r,s) \mathcal{W}_{\psi} f_{\epsilon}(r,s) ds \dfrac{ dr}{|r|^{d+1}}= \dfrac{1}{C_{\psi}} \int_{0<|r|\leq 1} \int_{\mathbb{R}^{d}} g(r,s) \mathcal{W}_{\psi} \chi_{\mathcal{R}}(r,s) ds \dfrac{ dr}{|r|^{d+1}}
    \end{equation*} and 
    \begin{equation*}
        \lim_{\epsilon\to 0}\dfrac{1}{C_{\psi}} \int_{|r|>1} \int_{\mathbb{R}^{d}} g(r,s) \mathcal{W}_{\psi} f_{\epsilon}(r,s) ds \dfrac{ dr}{|r|^{d+1}}=\dfrac{1}{C_{\psi}} \int_{|r|> 1} \int_{\mathbb{R}^{d}} g(r,s) \mathcal{W}_{\psi} \chi_{\mathcal{R}}(r,s) ds \dfrac{ dr}{|r|^{d+1}}.
    \end{equation*} 
    Then
    \begin{align*}
        \lim_{\epsilon\to 0}\dfrac{1}{C_{\psi}} \int_{\mathbb{R}^{\times}} \int_{\mathbb{R}^{d}} g(r,s) \mathcal{W}_{\psi} f_{\epsilon}(r,s) ds \dfrac{ dr}{|r|^{d+1}}&= \lim_{\epsilon\to 0}\dfrac{1}{C_{\psi}} \int_{0<|r|\leq 1} \int_{\mathbb{R}^{d}} g(r,s) \mathcal{W}_{\psi} f_{\epsilon}(r,s) ds \dfrac{ dr}{|r|^{d+1}}\\ & + \lim_{\epsilon\to 0}\dfrac{1}{C_{\psi}} \int_{|r|>1} \int_{\mathbb{R}^{d}} g(r,s) \mathcal{W}_{\psi} f_{\epsilon}(r,s) ds \dfrac{ dr}{|r|^{d+1}}\\
        &= \dfrac{1}{C_{\psi}} \int_{0<|r|\leq 1} \int_{\mathbb{R}^{d}} g(r,s) \mathcal{W}_{\psi} \chi_{\mathcal{R}}(r,s) ds \dfrac{ dr}{|r|^{d+1}}\\ & + \dfrac{1}{C_{\psi}} \int_{|r|> 1} \int_{\mathbb{R}^{d}} g(r,s) \mathcal{W}_{\psi} \chi_{\mathcal{R}}(r,s) ds \dfrac{ dr}{|r|^{d+1}}\\ &= \dfrac{1}{C_{\psi}} \int_{\mathbb{R}^{\times}} \int_{\mathbb{R}^{d}} g(r,s) \mathcal{W}_{\psi} \chi_{\mathcal{R}}(r,s) ds \dfrac{ dr}{|r|^{d+1}}.
    \end{align*}  Finally, using Proposition \ref{proplimit} we conclude that 
    \begin{equation*}
         \mu(\mathcal{R})=\dfrac{1}{C_{\psi}} \int_{\mathbb{R}^{\times}} \int_{\mathbb{R}^{d}} g(r,s) \mathcal{W}_{\psi} \chi_{\mathcal{R}}(r,s) ds \dfrac{ dr}{|r|^{d+1}}.
    \end{equation*}
This completes the proof.
\end{proof}
We provide an example that reconstructs the measure $\mu$ associated to a function $g\in \mathcal{F}_{\psi}$ using Proposition \ref{passinglimit}:
\begin{example}\label{eg3}
        Choose $f\in C_{c}^{2}(\mathbb{R}).$ Define 
        \begin{equation*}
            g(r,s)= \mathcal{W}_{\psi} f(r,s)
        \end{equation*} for all $(r,s)\in \mathbb{R}^{\times} \times \mathbb{R}.$ We show that $g\in \mathcal{F}_{\psi}$ as follows; Let $g= \mathcal{W}_{\psi} f$ for some $f\in C_{c}^{2}(\mathbb{R})$.  It is well-known that $g$ satisfies the condition $(a)$ of Theorem \ref{main:thm1} (See \cite[Section 2.5]{daubechies1992ten}). Next, we establish condition $(b)$ of Theorem \ref{main:thm1}. It is easy to see that for any $h\in C_{c}^{2}(\mathbb{R})$
            \begin{align*}
        \int_{\mathbb{R}^{\times}} \int_{\mathbb{R}} \int_{\mathbb{R}} \left| f(x) \psi_{r,s}(x) \mathcal{W}_{\psi} h(r,s)\right| dx ds \dfrac{ dr}{r^{2}}&\leq \|\psi\|_{\infty} \int_{\mathbb{R}^{\times}} \int_{\mathbb{R}} \int_{\mathbb{R}} \dfrac{1}{|r|^{\frac{1}{2}}} |f(x)| \left|\mathcal{W}_{\psi} h (r,s)\right| dx ds \dfrac{ dr}{r^{2}}\\
                &\leq \|\psi\|_{\infty} \int_{\mathbb{R}^{\times}} \int_{\mathbb{R}} \left[\int_{\mathbb{R}} |f(x)| dx\right] \dfrac{1}{|r|^{\frac{1}{2}}} \left|\mathcal{W}_{\psi} h (r,s)\right|ds\dfrac{ dr}{r^{2}}\\
                &= \|\psi\|_{\infty} \|f\|_{1} \int_{\mathbb{R}^{\times}} \int_{\mathbb{R}} \dfrac{1}{|r|^{\frac{1}{2}}} \left|\mathcal{W}_{\psi} h (r,s)\right| ds \dfrac{ dr }{r^{2}}\\&< \infty \quad (\mbox{Using Lemma \ref{main:lemma1}}).
            \end{align*} Then
            \begin{align*}
           \dfrac{1}{C_{\psi}}     \left|\int_{\mathbb{R}^{\times}} \int_{\mathbb{R}} \mathcal{W}_{\psi} f(r,s) \mathcal{W}_{\psi} h(r,s) ds \dfrac{ d r}{r^{2}} \right|&= \dfrac{1}{C_{\psi}}    \left|\int_{\mathbb{R}^{\times}} \int_{\mathbb{R}} \left[\int_{\mathbb{R}} \overline{\psi_{r,s}(x)} f(x) dx\right] \mathcal{W}_{\psi} h(r,s)\right|ds\dfrac{ dr}{r^{2}}\\
                &= \left| \int_{\mathbb{R}} f(x) \left[\dfrac{1}{C_{\psi}} \int_{\mathbb{R}^{\times}} \int_{\mathbb{R}} \overline{\psi_{r,s}(x)} \mathcal{W}_{\psi} h(r,s) ds \dfrac{ dr}{r^{2}}\right] dx\right| \left(\mbox{By Fubini-Tonelli}\right)\\
                & =\left| \int_{\mathbb{R}} f(x) h(x) dx\right| \quad \left(\mbox{By Lemma \ref{reconstruction lemma}} \right)\\
                &\leq \int_{\mathbb{R}} |f(x)| |h(x)| dx\\
                &\leq \|f\|_{1} \|h\|_{\infty}.
            \end{align*} Now observe that 
        \begin{align*}
          \left|\mathcal{W}_{\psi} f(r,s)\right| &= \dfrac{1}{|r|^{\frac{1}{2}}}  \left| \int_{\mathbb{R}} \overline{\psi\left(\dfrac{x-s}{r}\right)} f(x)dx\right|  \\
            &= \left|{|r|^{\frac{1}{2}}} \int_{\mathbb{R}} \overline{\psi(u)} f(ru+s)du\right| \quad \left(\mbox{Put $u= \dfrac{x-s}{r}$}\right)\\
            &= \left| {|r|^{\frac{1}{2}}} \int_{\mathbb{R}}  \overline{\psi(u)} \left[ f(s)+ f'(s) ru + f''(r\theta u + s) r^{2} u^{2} \right] du\right| \quad \left(\mbox{Apply Taylor's theorem to $f$}\right)\\
            &= \left|{|r|^{\frac{5}{2}}} \int_{\mathbb{R}} u^{2} \overline{\psi(u)} f''(r\theta u+ s) du \right| \quad \left(\mbox{Using vanishing moments}\right).
        \end{align*} Integrating with respect to $ds$ gives that 
        \begin{align*}
            \int_{\mathbb{R}} \left|\mathcal{W}_{\psi} f(r,s)\right| ds&\leq |r|^{\frac{5}{2}} \int_{\mathbb{R}} \int_{\mathbb{R}} u^{2} |{\psi(u)}|| f''(r\theta u+ s)| du ds\\
            &= |r|^{\frac{5}{2}} \int_{\mathbb{R}} u^{2} |\psi(u)| \left[\int_{\mathbb{R}} | f''(r\theta u+ s)| ds\right] du \quad \left(\mbox{By Fubini-Tonelli theorem}\right)\\
            &= |r|^{\frac{5}{2}} \int_{\mathbb{R}} u^{2} |\psi(u)| \left[\int_{\mathbb{R}} | f''( z)| dz\right] du \quad \left(\mbox{Take $z=r\theta u+ s $}\right)\\
            &= |r|^{\frac{5}{2}} \|f''\|_{1} \int_{\mathbb{R}} u^{2} |\psi(u)| du.
        \end{align*} Set $C= \|f''\|_{1} \displaystyle \int_{\mathbb{R}} u^{2} |\psi(u)| du$. Integrating both sides with respect to $\dfrac{dr}{|r|^{\frac{3}{2}}}$ over $\{r\in \mathbb{R}^{\times}: 0<|r|\leq 1\}$ yields that 
        \begin{align*}
            \int_{0<|r|\leq 1} \left[\int_{\mathbb{R}} \left|\mathcal{W}_{\psi} f(r,s)\right| ds\right] \dfrac{dr}{|r|^{\frac{3}{2}}}\leq C\int_{0<|r|\leq 1} |r| dr =C<\infty.
        \end{align*} This then satisfies the required integrability assumption in Equation \ref{refined assumption}. Using Proposition \ref{passinglimit} we get,
         $$ \mu([a,b])= \dfrac{1}{C_{\psi}}     \int_{\mathbb{R}^{\times}} \int_{\mathbb{R}} g(r,s) \mathcal{W}_{\psi} \chi_{[a,b]}(r,s) ds \dfrac{ dr}{r^{2}}.$$
        Observe that 
            \begin{align*}
              \int_{\mathbb{R}^{\times}} \int_{\mathbb{R}} \int_{\mathbb{R}} \left|\mathcal{W}_{\psi} f(r,s)\right| |\chi_{[a,b]}(x)| |\psi_{r,s}(x)|  dx ds\dfrac{ dr}{r^{2}}&= \int_{\mathbb{R}^{\times}} \int_{\mathbb{R}} \left|\mathcal{W}_{\psi} f(r,s)\right|\left[\int_{a}^{b} |\psi_{r,s}(x)| dx\right]ds  \dfrac{ dr}{r^{2}}\\
                &\leq \|\psi\|_{\infty}(b-a) \int_{\mathbb{R}^{\times}} \int_{\mathbb{R}} \dfrac{1}{|r|^{\frac{1}{2}}} \left|\mathcal{W}_{\psi} f(r,s)\right| ds \dfrac{ dr}{r^{2}}\\
                & < \infty.
            \end{align*} The last finiteness follows from Lemma \ref{main:lemma1} which helps us use Fubini-Tonelli's theorem in the following estimate,
            \begin{align*}
          \dfrac{1}{C_{\psi}}      \int_{\mathbb{R}^{\times}} \int_{\mathbb{R}} \mathcal{W}_{\psi} f(r,s) \mathcal{W}_{\psi} \chi_{[a,b]}(r,s) ds \dfrac{ dr}{r^{2}}&=   \dfrac{1}{C_{\psi}}  \int_{\mathbb{R}^{\times}} \int_{\mathbb{R}} \mathcal{W}_{\psi} f(r,s) \left[\int_{\mathbb{R}} \chi_{[a,b]}(x) \overline{\psi_{r,s}}(x) dx\right] ds \dfrac{ dr}{r^{2}}\\
                &= \int_{\mathbb{R}} \chi_{[a,b]}(x) \left[\dfrac{1}{C_{\psi}} \int_{\mathbb{R}^{\times}} \int_{\mathbb{R}} \mathcal{W}_{\psi} f(r,s) \overline{\psi_{r,s}}(x) ds \dfrac{ dr}{r^{2}}\right] dx (\mbox{Fubini-Tonelli})\\
                &=\int_{\mathbb{R}} \chi_{[a,b]}(x) f(x) dx \quad \left(\mbox{Using Lemma \ref{reconstruction lemma}} \right) \\
                &= \int_{a}^{b} f(x) dx.
            \end{align*} As a consequence, $ \mu([a,b])= \displaystyle \int_{a}^{b} f(x) dx.$
           \begin{remark}
               \begin{enumerate}
                   \item We remark that the Example \ref{eg3} can be extended to higher dimensions in a similar manner to evaluate the measure of a closed rectangle in $\mathbb{R}^{d}.$
                   \item Given a bounded, compactly supported function $f$ in $L^1(\mathbb{R})$, we can use the wavelet inversion formula for $f$ a.e., as in \cite[Proposition 2.4.2]{daubechies1992ten} instead of our use of Lemma \ref{reconstruction lemma}, to get the same result as in the Example \ref{eg3} above. Namely, if we choose $g\in \mathcal{F}_{\psi}$ such that $g(r, s)= \mathcal{W}_{\psi} f(r, s)$, then $\mu([a, b])=\displaystyle \int_a^bf(x) dx$. 
                   \item For absolutely continuous measures, convergence of $\mu_n= f_ndx$ to $\mu= fdx$ in total variation norm in $\mathcal{M}(\mathbb{R}^d)$ is equivalent to the convergence of the functions $f_n$ to $f$ in $L^1(\mathbb{R}^d)$. Proposition \ref{main;thm3}, along with the above Example \ref{eg3}, will imply convergence of the continuous wavelet transforms, $\mathcal{W}_{\psi}(f_n)$ to $\mathcal{W}_{\psi}(f)$ in $L^p_{w_p}(\mathbb{R^{\times}}\times\mathbb{R}^d, \frac{drds}{|r|^{d+1}})$ with specific weights $w_p$ chosen as in Proposition \ref{main;thm3}.         
               \end{enumerate}
           \end{remark}

    \end{example} 
    The following example suggests that the integrability assumption given in Equation \ref{refined assumption} cannot be dropped from Proposition \ref{passinglimit}.
\begin{example} \label{eg4}

    Let $\psi$ denote a wavelet satisfying the Assumption \ref{assumptions} and take $\mu= \delta_{0}$. Let $g= \mathcal{W}_{\psi} \mu$. We check that Equation \ref{refined assumption} holds for this $g$ or not. Note that 
        \begin{align*}
            \int_{\mathbb{R}} |r|^{-\frac{1}{2}} \left|\psi\left(-\dfrac{s}{r}\right)\right|ds&= |r|^{\frac{1}{2}} \int_{\mathbb{R}} |\psi(u)| du \quad \left(\mbox{Take $u=-\dfrac{s}{r}$}\right)\\
            &= |r|^{\frac{1}{2}} \|\psi\|_{1}.
        \end{align*} This shows that $G(r)= |r|^{\frac{1}{2}} \|\psi\|_{1}$. Then 
        \begin{align*}
            \int_{0<|r|\leq 1} |r|^{-\frac{3}{2}}  G(r)dr =\|\psi\|_{1} \int_{0<|r|\leq 1}  |r|^{\frac{1}{2}} |r|^{-\frac{3}{2}}  dr = \|\psi\|_{1} \int_{0<|r|\leq 1} |r|^{-1} dr.
        \end{align*} The last integral is divergent.

 Let $\psi: \mathbb{R}\to \mathbb{R}$ be the function given in Example \ref{eg0}. We show that the integral 
        \begin{equation*}
            \int_{\mathbb{R}^{\times}} \int_{\mathbb{R}} \left|g(r,s) \mathcal{W}_{\psi} \chi_{[0,1]}(r,s) \right|ds \dfrac{dr}{r^2}
        \end{equation*} is divergent. Observe that 
        \begin{align*}
            \int_{\mathbb{R}} \left|g(r,s) \mathcal{W}_{\psi} \chi_{[0,1]}(r,s) \right|ds&= r^{-1} \int_{\mathbb{R}} \left|\psi \left(-\dfrac{s}{r}\right)\right| \left| \int_{0}^{1} \psi\left(\dfrac{x-s}{r}\right)dx\right|ds\\
            &=\int_{\mathbb{R}} \left|\psi(-u)\right| \left| \int_{0}^{1} \psi \left( \dfrac{x}{r}-u\right)dx\right|du \quad \left(\mbox{Put $u=-\dfrac{s}{r}$}\right)\\
            &= r \int_{\mathbb{R}} \left|\psi(-u)\right| \left|\int_{-u}^{\frac{1}{r}-u} \psi(t) dt\right| du \quad \left(\mbox{Take $t= \dfrac{x}{r}-u$}\right).
        \end{align*} Choose $0<r<\dfrac{1}{2}$. Since supp$(\psi) \subset [-1,1]$, then for all $|t|\leq 1$ and $u\in [-1,1]$, we have $\dfrac{1}{r}-u\geq 1$. Thus,
        \begin{equation*}
            \int_{-u}^{\frac{1}{r}-u} \psi(t) dt= \int_{-u}^{1} \psi(t) dt.
        \end{equation*} Let 
        \begin{equation*}
            A= \int_{-1}^{1} \left|\psi(-u)\right| \left|\int_{-u}^{1} \psi(t) dt\right| du.
        \end{equation*} 
        Then
        \begin{equation*}
           \int_{\mathbb{R}} \left|g(r,s) \mathcal{W}_{\psi} \chi_{[0,1]}(r,s) \right|ds= r A.
        \end{equation*} 
        For $0< u< \dfrac{1}{3}$, we have $|\psi(-u)|= |\psi(u)|>0$. Therefore, $\displaystyle \left|\int_{-u}^{1} \psi(t) dt\right| >0$ and hence $A>0$. Now 
        \begin{align*}
           \int_{\mathbb{R}^{\times}} \int_{\mathbb{R}} \left|g(r,s) \mathcal{W}_{\psi} \chi_{[0,1]}(r,s) \right|ds \dfrac{dr}{r^2} &\geq \int_{0}^{\frac{1}{2}} \int_{\mathbb{R}} \left|g(r,s) \mathcal{W}_{\psi} \chi_{[0,1]}(r,s) \right|ds \dfrac{dr}{r^2}\\
           &=A \int_{0}^{\frac{1}{2}} \dfrac{1}{r} dr.
        \end{align*} Thus, the integral on the left hand side diverges, and therefore not equal to $\delta_0([0, 1])$ unlike as in Example \ref{eg2}.
    \end{example}

\section{Concluding Remarks}
Our study has been confined to wavelet transform of complex Radon measures on the Euclidean space, namely, $\mathcal{M}(\mathbb{R}^d)$ with specific restrictions on the analyzing wavelet. In particular, we have established necessary and sufficient conditions for a measurable function on $\mathbb{R}^{\times }\times \mathbb{R}^{d}$ to be the continuous wavelet transform of a complex Radon measure on $\mathbb{R}^{d}$
under suitable regularity and vanishing-moment assumptions on the analyzing wavelet. We showed that the range of continuous wavelet transform of complex Radon measures, $\mathcal{F}_{\psi}$, with a suitable norm, is isometrically isomorphic to the space $\mathcal{M}(\mathbb{R}^d)$ endowed with the total variation norm and inherits a Banach algebra structure. Using the assumptions on $\psi$ we embed $\mathcal{M}(\mathbb{R}^d)$ in some family of weighted $L^p$ spaces, $1\leq p< \infty$. We also obtained a reconstruction formula to compute the value of the measure of closed rectangles from a given function $g$ in $\mathcal{F}_{\psi}$ through a limiting process. Next we provided a direct approach using continuous wavelet transform of characteristic functions, under an additional integrability assumption on the measurable function $g$. Further studies on the application of the new norm on the range of the wavelet transform, as introduced in \ref{norm} can be initiated. One direction of further work would be to relax the restrictions on the wavelet. Another direction would be to try and extend the formulation to that of Locally Compact Abelian groups. 

\section{Acknowledgments}
We thank the reviewers for their useful comments and suggestions. Vaishakh Jayapraksh is supported by the INSPIRE Ph.D Fellowship of the Department of Science and Technology, Government of India. Both the authors want to sincerely thank Prof. M.N.N. Namboodiri, retired Professor, Department of Mathematics, CUSAT. We are highly indebted to him for his deep insights and suggestions which guided and inspired us throughout this work.
\section{Declaration on the Use of AI Tools}
During the preparation of this work, the authors used Open AI for language polishing, reference checking, and to identify typographical and expository errors. No mathematical content, proofs, or derivations were generated by AI tools. The authors reviewed and verified the entire manuscript and take full responsibility for the content of this publication.
\bibliographystyle{abbrv}
	\nocite{*}
\bibliography{Wavelet_Measure}

@book{daubechies1992ten,
 author = {Daubechies, Ingrid},
 title = {Ten lectures on wavelets},
 fseries = {CBMS-NSF Regional Conference Series in Applied Mathematics},
 series = {CBMS-NSF Reg. Conf. Ser. Appl. Math.},
 volume = {61},
 isbn = {0-89871-274-2},
 year = {1992},
 publisher = {Philadelphia, PA: SIAM},
 language = {English},
 doi = {10.1137/1.9781611970104},
 zbMATH = {53691},
 Zbl = {0776.42018}
}

@book{folland1999real,
 author = {Folland, Gerald B.},
 title = {Real analysis. {Modern} techniques and their applications.},
 edition = {2nd ed.},
 fseries = {Pure and Applied Mathematics. A Wiley-Interscience Series of Texts, Monographs and Tracts},
 series = {Pure Appl. Math., Wiley-Intersci. Ser. Texts Monogr. Tracts},
 isbn = {0-471-31716-0},
 year = {1999},
 publisher = {New York, NY: Wiley},
 language = {English},
 zbMATH = {1324223},
 Zbl = {0924.28001}
}

@article{ashurov2011continuous,
  title={On continuous wavelet transforms of distributions},
  author={Ashurov, Ravshan and Butaev, Almaz},
  journal={Applied mathematics letters},
  volume={24},
  number={9},
  pages={1578--1583},
  year={2011},
  publisher={Elsevier}
}

@article{pathak2004wavelet,
  title={The wavelet transform of distributions},
  author={Pathak, Ram S},
  journal={Tohoku Mathematical Journal, Second Series},
  volume={56},
  number={3},
  pages={411--421},
  year={2004},
  publisher={Mathematical Institute, Tohoku University}
}

@article{grossmann1985transforms,
  title={Transforms associated to square integrable group representations. I. General results},
  author={Grossmann, Alex and Morlet, Jean and Paul, T},
  journal={Journal of Mathematical Physics},
  volume={26},
  number={10},
  pages={2473--2479},
  year={1985},
  publisher={American Institute of Physics}
}

@book{holschneider1995wavelets,
  author    = {Holschneider, Matthias},
  title     = {Wavelets: An Analysis Tool},
  series    = {Oxford Mathematical Monographs},
  publisher = {Clarendon Press},
  address   = {Oxford},
  year      = {1995},
  isbn      = {978-0198534815}
}

@book{Meyer1992Wavelets,
  author    = {Meyer, Yves},
  title     = {Wavelets and Operators},
  publisher = {Cambridge University Press},
  year      = {1992},
  series    = {Cambridge Studies in Advanced Mathematics},
  volume    = {37},
  translator = {Salinger, D. H.},
  address   = {Cambridge, UK},
  isbn      = {9780521420006}
}

@book{triebel2006theory,
 author = {Triebel, Hans},
 title = {Theory of function spaces. {III}},
 fseries = {Monographs in Mathematics},
 series = {Monogr. Math., Basel},
 issn = {1017-0480},
 volume = {100},
 isbn = {3-7643-7581-7},
 year = {2006},
 publisher = {Basel: Birkh{\"a}user},
 language = {English},
 zbMATH = {5044786},
 Zbl = {1104.46001}
}

@article{saneva2010wavelet,
  title={Wavelet expansions and asymptotic behavior of distributions},
  author={Saneva, Katerina and Vindas, Jasson},
  journal={Journal of Mathematical Analysis and Applications},
  volume={370},
  number={2},
  pages={543--554},
  year={2010},
  publisher={Elsevier}
}

@article{rieder1990wavelet,
  title={The wavelet transform on Sobolev spaces and its approximation properties},
  author={Rieder, Andreas},
  journal={Numerische Mathematik},
  volume={58},
  number={1},
  pages={875--894},
  year={1990},
  publisher={Springer}
}

@article{pandey2019continuous,
 author = {Pandey, J. N.},
 title = {Continuous wavelet transform of {Schwartz} distributions},
 fjournal = {Rocky Mountain Journal of Mathematics},
 journal = {Rocky Mt. J. Math.},
 issn = {0035-7596},
 volume = {49},
 number = {6},
 pages = {2005--2028},
 year = {2019},
 language = {English},
 doi = {10.1216/RMJ-2019-49-6-2005},
 zbMATH = {7136591},
 Zbl = {1428.42074}
}

@article{zbMATH05025041,
 author = {Saneva, Katerina and Bu{\v{c}}kovska, Aneta},
 title = {Asymptotic expansion of distributional wavelet transform},
 fjournal = {Integral Transforms and Special Functions},
 journal = {Integral Transforms Spec. Funct.},
 issn = {1065-2469},
 volume = {17},
 number = {2-3},
 pages = {85--91},
 year = {2006},
 language = {English},
 doi = {10.1080/10652460500436551},
 zbMATH = {5025041},
 Zbl = {1101.46025}
}

@article{zbMATH02233272,
 author = {Saneva, Katerina and Bu{\v{c}}kovska, Aneta},
 title = {Asymptotic behaviour of the distributional wavelet transform at 0},
 fjournal = {Mathematica Balkanica. New Series},
 journal = {Math. Balk., New Ser.},
 issn = {0205-3217},
 volume = {18},
 number = {3-4},
 pages = {437--441},
 year = {2004},
 language = {English},
 zbMATH = {2233272},
 Zbl = {1077.42032}
}

@article{pandey2019continuous1,
  title={Continuous wavelet transform of Schwartz tempered distributions},
  author={Pandey, JN and Upadhyay, SK},
  journal={Cogent Mathematics \& Statistics},
  volume={6},
  number={1},
  pages={1623647},
  year={2019},
  publisher={Taylor \& Francis}
}

@article{holschneider1994fractal,
 author = {Holschneider, Matthias},
 title = {Fractal wavelet dimensions and localization},
 fjournal = {Communications in Mathematical Physics},
 journal = {Commun. Math. Phys.},
 issn = {0010-3616},
 volume = {160},
 number = {3},
 pages = {457--473},
 year = {1994},
 language = {English},
 doi = {10.1007/BF02173424},
 zbMATH = {553838},
 Zbl = {0792.28009}
}

@article{antoine19992,
  title={The 2-D wavelet transform, physical applications and generalizations},
  author={Antoine, JP},
  journal={Wavelets in physics},
  pages={23--75},
  year={1999},
  publisher={Cambridge University Press}
}

@article{rao1994application,
  title={Application of Carleson’s theorem to wavelet inversion},
  author={Rao, Murali and SikiC, Hrvoje and Song, Renming},
  journal={Control Cybern},
  volume={23},
  pages={761--771},
  year={1994}
}

@book{folland2016course,
 author = {Folland, Gerald B.},
 title = {A course in abstract harmonic analysis},
 edition = {2nd updated edition},
 fseries = {Textbooks in Mathematics},
 series = {Textb. Math.},
 isbn = {978-1-4987-2713-6; 978-1-032-92221-8; 978-1-4987-2715-0},
 year = {2016},
 publisher = {Boca Raton, FL: CRC Press},
 language = {English},
 doi = {10.1201/b19172},
 zbMATH = {6473590},
 Zbl = {1342.43001}
}

@Inbook{Pathak2009,
author="Pathak, Ram Shankar",
title="TheWavelet Transform on Lp",
bookTitle="The Wavelet Transform",
year="2009",
publisher="Atlantis Press",
address="Paris",
pages="21--48",
isbn="978-94-91216-24-4",
doi="10.2991/978-94-91216-24-4_2",
url="https://doi.org/10.2991/978-94-91216-24-4_2"
}
\end{document}